%% file: ClumpsNStuff.tex
\documentclass[a4paper,10pt]{amsart}
\usepackage{standalone}
\usepackage{amsfonts}
\usepackage{amsthm}
\usepackage{amssymb}
\usepackage{amsmath}
\usepackage{theoremref}
\usepackage{enumerate}
\usepackage{bbm}
\usepackage{bm}
\usepackage{pgfplots}
\pgfplotsset{compat=1.18} 
\usepgfplotslibrary{fillbetween}
\usetikzlibrary{intersections}
\usepackage{mathtools}
\usetikzlibrary{patterns}

\usepackage{geometry}
\usepackage{a4wide}

\numberwithin{equation}{section}

\renewcommand{\L}{\mathcal{L}}

\newcommand{\res}{\text{res}}

\newcommand{\T}{\mathbb{T}}

\newcommand{\conj}[1]{\overline{#1}}
\newcommand{\D}{\mathbb{D}}

\newcommand{\R}{\mathbb{R}}
\renewcommand{\H}{\mathbb{H}}
\newcommand{\Po}{\mathcal{P}}

\newcommand{\dist}[2]{\text{dist}( #1, #2 ) }

\newcommand{\m}{\textit{m}}
\newcommand{\hil}{\mathcal{H}}

\newcommand{\hb}{\mathcal{H}(b)}

\renewcommand\Re{\operatorname{Re}}
\renewcommand\Im{\operatorname{Im}}

\newtheorem{mainthm}{Theorem}

\newtheorem{thm}{Theorem}[section]
\newtheorem*{thm*}{Theorem}
\newtheorem{lem}[thm]{Lemma}

\newtheorem{cor}[thm]{Corollary}
\newtheorem*{cor*}{Corollary}
\newtheorem{prop}[thm]{Proposition}
\theoremstyle{definition}

\theoremstyle{definition}
\newtheorem{remark}[thm]{Remark}

\newtheorem{claim*}{Claim}

\title[Spectral clumping for functions decreasing rapidly on a half-line]{Spectral clumping for functions and distributions decreasing rapidly on a half-line}
\author{Bartosz Malman}

\address{Division of Mathematics and Physics, 
        Mälardalen University,
		Västerås, Sweden}
\email{bartosz.malman@mdu.se} 

\begin{document}

\begin{abstract} We demonstrate a phenomenon of condensation of the Fourier transform $\widehat{f}$ of a function $f$ defined on the real line $\R$ which decreases rapidly on one half of the line. For instance, we prove that if $f$ is square-integrable on $\R$, then a one-sided estimate of the form \[\rho_f(x) := \int_x^{\infty} |f(t)| \,dt = \mathcal{O}\big(e^{-c\sqrt{x}} \big), \quad x > 0\] for some $c > 0$, forces the non-zero frequencies $\sigma(f) := \{ \zeta \in \R : |\widehat{f}(\zeta)| > 0 \}$ to clump: this set differs from an open set $U$ only by a set of Lebesgue measure zero, and $\log |\widehat{f}|$ is locally integrable on $U$. In particular, if $f$ is non-zero, then there exists an interval on which $\log |\widehat{f}|$ is integrable. The roles of $f$ and $\widehat{f}$ above may be interchanged, and the result extends also to a large class of tempered distributions. We show that the above decay condition is close to optimal, in the following sense: a non-zero entire function $f$ exists which is square-integrable on $\R$, for which $\sigma(f)$ is a subset of a compact set $E$ containing no intervals, and for which the estimate $\rho_f(x) = \mathcal{O}\big( e^{-x^a}\big)$, $x > 0$, holds for every $a \in (0, 1/2)$.
\end{abstract}

\maketitle

\section{Introduction}

\subsection{Fourier transform, its support and size}

This note studies a certain manifestation of the \textit{uncertainty principle in Fourier analysis}, where a smallness condition on a function $f$ forces its Fourier transform $\widehat{f}$ to be, in some sense, large. Vice versa, smallness of $\widehat{f}$ forces $f$ to be large. In our context, the smallness is defined in terms of a one-sided decay condition, and the largeness in terms of the existence of a \textit{clump}. This will be our moniker for an interval on which the function has an integrable logarithm. We emphasize that our results concern functions with a spectrum which might vanish on an interval (commonly referred to as functions with a \textit{spectral gap}), but for which the spectrum should be large on some other interval. 

We will use the following definition of the transform:
\begin{equation}
    \label{FourierTransformDef}
    \widehat{f}(\zeta) := \int_\R f(x) e^{-ix\zeta}\,d\lambda(x), \quad \zeta \in \R.
\end{equation} Here $d\lambda(x) = dx/\sqrt{2\pi}$ is a normalization of the Lebesgue measure $dx$ on $\R$. Then, the inversion formula is given by
\begin{equation}
    \label{InverseFourierTransformDef}
    f(x) := \int_\R \widehat{f}(\zeta) e^{i\zeta x }\,d\lambda(\zeta), \quad x \in \R.
\end{equation} 
For $p > 0$, let $\L^p(\R, dx)$ be the usual Lebesgue space of functions $f$ for which $|f|^p$ is integrable with respect to $dx$. The formula \eqref{FourierTransformDef} can be interpreted literally only for $f \in \L^1(\R, dx)$. It is interpreted in terms of Plancherel's theorem in the case $f \in \L^2(\R, dx)$, and in order to state our most general results we will later need to interpret the transform in the sense of distribution theory. The \textit{spectrum} $\sigma(f)$ of a function $f$ is the subset of $\R$ on which $\widehat{f}$ lives. Since $\widehat{f} \in \L^2(\R, dx)$ is defined only up to a set of Lebesgue measure zero, so is the spectrum $\sigma(f)$ in this case. If we accept making errors of measure zero (which we will), we may define the spectrum as \[ \sigma(f) := \{ \zeta \in \R : |\widehat{f}(\zeta)| > 0 \}, \quad f \in \L^1(\R, dx) \cup \L^2(\R, dx).\] Note specifically that our definition of $\sigma(f)$ might not coincide with the usual notion of closed \textit{support} of the distribution $\widehat{f}$.

The uncertainty principle in Fourier analysis presents itself in plenty of ways, and the excellent monograph \cite{havinbook} of Havin and Jöricke describes many of its most interesting interpretations. One of them is the following statement, well-known to function theorists. If $f \in \L^2(\R, dx)$ is non-zero and $\R_-$ is the negative half-axis, then we have the implication
\begin{equation}
    \label{JensenIneq} f(x) \equiv 0 \text{ on } \R_- \quad \Rightarrow \quad \int_{\R} \frac{\log |\widehat{f}(\zeta)|}{1+\zeta^2} \, d\zeta  > -\infty.
\end{equation} Here the extreme decay (indeed, vanishing) of $f$ on a half-axis implies global integrability of $\log |\widehat{f}|$ against the Poisson measure $d\zeta / (1+\zeta^2)$. A fortiori, $\log |\widehat{f}|$ is integrable on every interval $I$ of $\R$. Naturally, this is not typical. By Plancherel's theorem, every function in $\L^2(\R, dx)$ is the Fourier transform of some other function in the same space. So on the other extreme, plenty of functions $f \in \L^2(\R, dx)$ have a Fourier transform which lives on sparse sets containing no intervals. This forces the divergence of the logarithmic integral of $\widehat{f}$ over any interval. In other words, plenty of functions admit no spectral clumps. The results of this note give conditions under which such clumps form.

\subsection{Condensation and sparseness of spectra and supports} We shall introduce our results at first in the context of the Hilbert space $\L^2(\R, dx)$. Here we can prove a claim which symmetric in $f$ and $\widehat{f}$, and also we can argue for near-optimality of the result. This is the content of \thref{CondensationTheorem} and \thref{SparsenessTheorem}. The more general distributional clumping result is presented in \thref{DistributionalClumpingTheorem}.

\begin{mainthm} \thlabel{CondensationTheorem} If $f \in \L^2(\R, dx)$ satisfies the estimate 
\begin{equation}
\label{rhofDecayCondThm}\rho_f(x) := \int_x^\infty |f(t)| \,dt = \mathcal{O}\big( e^{-c\sqrt{x}} \big), \quad x > 0 \end{equation} for some constant $c > 0$, then there exists an open set $U$ which coincides with $\sigma(f)$ up to a set of Lebesgue measure zero, and for every $x \in U$ there exists an interval $I$ containing $x$ such that \[ \int_I \log |\widehat{f}(t)| \, dt > -\infty. \]
\end{mainthm}

In other words, the one-sided decay condition \eqref{rhofDecayCondThm} implies that $\widehat{f}$ lives on the union of the spectral clumps of $f$. Since the Fourier transform is a unitary operation on $\L^2(\R, dx)$, the roles of $f$ and $\widehat{f}$ may obviously be interchanged in the statement of \thref{CondensationTheorem}. Thus a one-sided spectral decay condition of $f$ implies local integrability properties of $\log |f|$ on the set where $f$ lives. In this form, the result encourages us to extend it to tempered distributions. We shall do so in a moment. 

The integrand in \eqref{rhofDecayCondThm} may seem a bit unnatural in the context of square-integrable functions $f$. It is more natural in the context of functions of tempered growth appearing in \thref{DistributionalClumpingTheorem}. Anyhow, we note that one can prove that an estimate of the form $\int_x^\infty |f(t)|^2 \, dt = \mathcal{O}\big( e^{-c\sqrt{x}} \big)$ in fact implies \eqref{rhofDecayCondThm} for some slightly smaller $c$.

We can prove also that the condition \eqref{rhofDecayCondThm} on the decay of $\rho_f$ appearing in \thref{CondensationTheorem} is close to optimal. We do so by exhibiting a non-zero function with rapid one-sided decay but sparse spectrum. 

\begin{mainthm} \thlabel{SparsenessTheorem}
For every $b > 0$, there exists a compact set $E \subset \R$ contained in $[0, b]$ which contains no intervals, and a non-zero entire function $f \in \L^2(\R, dx)$ which satisfies \[ \rho_f(x) = \mathcal{O}\big( e^{-x^a} \big), \quad x > 0 \] for every $a \in (0, 1/2)$, and such that $\sigma(f)$ is contained within $E$.
\end{mainthm}

After an initial reduction, the proof of this result follows ideas of Khrushchev from \cite{khrushchev1978problem}. Note that the function $f$ appearing in \thref{SparsenessTheorem} is entire by the virtue of having a spectrum $\sigma(f)$ of compact support. More importantly, the condition on $E$ implies that $I \setminus E$ has positive Lebesgue measure for every interval $I$, so we obtain \[ \int_I \log |\widehat{f}(t)| \, dt = -\infty\] for every interval $I \subset \R$. This is in contrast to the conclusion of \thref{CondensationTheorem}. It follows that the exponent $a = 1/2$ in estimates of the form $\rho_f(x) = \mathcal{O}\big( e^{-x^a} \big)$ is critical for the spectral clumping phenomenon.

As mentioned above, clumping statements makes sense for objects in a class much wider than $\L^2(\R, dx)$. Here is our distrbutional result.

\begin{mainthm} \thlabel{DistributionalClumpingTheorem} Let $f$ be a tempered distribution on $\R$ which is a measurable function satisfying \[ \int_{\R} \frac{|f(x)|}{(1+|x|)^n} \, dx < \infty\] for some $n > 0$. If the distributional Fourier transform $\widehat{f}$ is an integrable function on some interval $[A, \infty)$, and the estimate $\rho_{\widehat{f}}(\zeta) = \mathcal{O}\big( e^{-c\sqrt{\zeta}}\big)$ holds for all sufficiently large positive $\zeta$, then there exists an open set $U$ such that $f$ vanishes almost everywhere outside of $U$, and for each $x \in U$ there exists an interval $I$ containing $x$ satisfying \begin{equation}
    \label{fClumpDist} \int_I \log |f(t)| \, dt > -\infty.
\end{equation} 
\end{mainthm}

For instance, the result shows that a function $f \in \L^1(\R, dx)$ which lives on a sparse set containing no intervals cannot satisfy even a one-sided spectral decay condition of the form $\rho_{\widehat{f}}(\zeta) \lesssim e^{-c\sqrt{\zeta}}$. Note also that in this extended form, our result includes the trivial but important examples such as $f = 1$ and $\widehat{f} = \delta_0$ (Dirac delta), the trigonometric functions and the polynomials. 

\subsection{A converse result} The Beurling-Malliavin theory implies a partial converse result. If $f$ is a locally integrable function on $\R$ which has a clump $I$ as in \eqref{fClumpDist}, and a constant $c > 0$ is given, then a bounded multiplier $m$ exists for which $mf$ has a Fourier transform satisfying $\rho_{\widehat{mf}}(\zeta) = \mathcal{O}\big(e^{-c\sqrt{\zeta}}\big)$ for $\zeta > 0$. To see this, recall that a smooth function $g$ supported in $I$ exists which satisfies the bilateral spectral decay $|\widehat{g}(\zeta)| \leq e^{-c\sqrt{|\zeta|}}$, $\zeta \in \R$ (this simpler version of the famous Beurling-Malliavin theorem is proved in \cite[p. 276-277]{havinbook}, and in fact we may ensure an even faster bilateral spectral decay of $g$). There exists also a bounded function $h \in \L^1(\R, dx)$ which satisfies $\sigma(h) \subset (0, \infty)$ and $|h(x)| = \min( |f(x)|, 1)$ on $I$ (we use the assumption that $I$ is a clump for $f$ and construct $h$ as in \eqref{hOuterDef} below). Then an argument similar to the one used in the proof of \thref{ConvolutionFourierDecayLemma} below shows that the function $\conj{h}g$ will satisfy the desired one-sided spectral decay, and clearly $\conj{h}g = mf$ for some bounded function $m$ supported in $I$.

\subsection{Clumping in other parts of analysis} The motivation for the research presented in this note was a desire to produce a self-contained exposition of the clumping phenomenon which was observed in two other contexts, both somewhat more esoteric than Fourier analysis on the real line. 

The first of these is a polynomial approximation problem in the unit disk $\D := \{ z \in \mathbb{C} : |z| < 1\}$. Here we are presented with a measure \[d\mu = G(1-|z|)dA(z) + w(z) d\m(z),\] where $dA$ and $d\m$ are the area and arc-length measures on $\D$ and $\T := \partial \D = \{ z \in \mathbb{C} : |z| = 1\}$. The functions $G$ and $w$ are non-negative weights, and one would like to understand under which conditions \textit{splitting} occurs. Namely, when is the weighted space $\L^2(\T, w \, d\m)$ contained in the closure of analytic polynomials in the $\L^2$-norm induced by the measure $\mu$? In the case that $G(1-|z|)$ decays exponentially as $|z| \to 1^-$, the necessary and sufficient condition is that $w$ has no clumps, or in other words that the integral of $\log w$ diverges over any arc on $\T$. The lack-of-clumping condition was conjectured by Kriete and MacCluer in \cite{kriete1990mean} and confirmed in \cite{malman2023revisiting}. Some of the techniques used in the proofs of the results in the present note are adaptations of the ideas from \cite{malman2023revisiting}.

The other context is a circle of ideas surrounding the \textit{Aleksandrov-Clark measures} appearing in spectral theory, and spaces $\hb$ defined by de Branges and Rovnyak, well-known to operator theorists. To any positive finite Borel measure $\nu$ on $\T$ we may associate a so-called \textit{Clark operator} $\mathcal{C}_\nu$ which takes a function $g \in \L^2(\T, d\nu)$ to the analytic function in $\D$ given by the formula \[\mathcal{C}_\nu g (z) := \frac{\int_\T \frac{g(x)}{1-\conj{x}z}d\nu(x)}{\int_\T \frac{1}{1-\conj{x}z} d\nu(x)}, \quad z \in \D.\] The operator $\mathcal{C}_\nu$ maps $\L^2(\T, d\nu)$ onto a space of analytic functions denoted by $\hb$, the symbol function $b: \D \to \D$ itself being related to $\nu$ by the formula \[\frac{1}{1-b(z)} = \int_\T \frac{1}{1-\conj{x}z} d\nu(x), \quad z \in \D\] in the case that $\nu$ is a probability measure, with a similar formula in the general case. For many choices of $\nu$ (or equivalently, choices of $b$), the space $\hb$ is somewhat mysterious, with the distinctive feature of containing very few functions extending analytically to a disk larger than $\D$. This extension property is characterized by the exponential decay of the Taylor series of the function, and the clumping of the absolutely continuous part of $\nu$ is decisive for existence and density of functions in $\hb$ which have a Taylor series decaying just a bit slower than exponentially. Results of this nature are contained in \cite{malman2023shift}. In fact, a Fourier series version of \thref{CondensationTheorem} is a consequence of the results in \cite{malman2023shift}.

\subsection{Other forms of the uncertainty principle} 
The implication \eqref{JensenIneq} has a well-known Fourier series version. If a function $f$ defined on the circle $\T := \{ z \in \mathbb{C} : |z| = 1\}$ is integrable with respect to arc-length $ds$ on $\T$, and the negative portion of the Fourier series of $f$ vanishes, then $\int_\T \log |f| ds > -\infty$, unless $f$ is the zero function.. Volberg derived the same conclusion from the weaker hypothesis of nearly-exponentially decaying negative portion of the Fourier series (see \cite{volberg1982logarithm} and the exposition in \cite{vol1987summability}). Work of Borichev and Volberg \cite{borichev1990uniqueness} contains related results.

The decay condition \eqref{rhofDecayCondThm} on $f \in \L^2(\R, dx)$ prohibits $\widehat{f}$ from living on a set $S$ containing no intervals. Somewhat related are uniqueness statements in which one seeks to give examples of pairs of sets $(E,S)$ for which the following implication is valid: if $f$ in a certain class lives on $E$ and $\widehat{f}$ lives on $S$, then $f \equiv 0$. One says that $(E,S)$ is then a \textit{uniqueness pair} for the corresponding class. A famous result of Benedicks presented in \cite{benedicks1985fourier} (see also \cite{amrein1977support}) says that $(E,S)$ is a uniqueness pair for integrable $f$ if both sets have finite Lebesgue measure, and the result holds not only for the real line $\R$ but also for the $d$-dimensional Euclidean space $\R^d$. Hedenmalm and Montes-Rodr\'iguez worked with the hyperbola $H = \{ (x,y) \in \R^2 : xy=1 \}$ and th class of finite Borel measures $\mu$ supported on $H$ which are absolutely continuous with respect to arclength on $H$. They proved in \cite{hedenmalm2011heisenberg} that if $\widehat{\mu}$ vanishes on certain types of discrete sets $\Lambda \subset \R^2$, then $\mu \equiv 0$, thus exhibiting interesting uniqueness pairs of the form $(H, \mathbb{\R}^2 \setminus \Lambda)$. Recent work of Radchenko and Viazovska on interpolation formulas for Schwartz functions in \cite{radchenko2019fourier} gives examples of pairs of discrete subsets $E$ and $S$ of $\R$ for which $(\R \setminus E, \R \setminus S)$ is a uniqueness pair for functions in the Schwartz class. Kulikov, Nazarov and Sodin exhibit similar interpolation formulas, and consequently new uniqueness pairs, in their recent work in \cite{kulikov2023fourier}.

\subsection{Notation} For a set $E \subset \R$ and a measure $\mu$ defined on $\R$, the space $\L^p(E, d\mu)$ denotes the usual Lebesgue space consisting of equivalence classes of functions living only on $E$ and satisfying the integrability condition $\int_E |f(x)|^p d\mu(x) < \infty$. The containment $\L^p(E, d\mu) \subset \L^p(\R, d\mu)$ is interpreted in the natural way. The symbols such as $dx$, $dt$ and $d\zeta$ denote the usual Lebesgue measure of the real line, while $d\lambda = dx/\sqrt{2\pi}$ will be the normalized version used in formulas involving Fourier transforms. If $E$ is a subset of $\R$, then $|E|$ denotes its usual Lebesgue measure. The positive half-axis of $\R$ is denoted by $\R_+ := \{ x \in \R: x \geq 0\}$, and we set also $\R_- := \R \setminus \R_+$. The notions of \textit{almost everywhere} and \textit{of measure zero} are always to be interpreted in the sense of Lebesgue measure on $\R$. The indicator function of a measurable set $E$ is denoted by $\mathbbm{1}_E$. Finally, we put $\log^+(x) := \max(\log(x), 0)$.

\section{Preliminaries}
\label{PreliminarySection}

Our proofs will use Hilbert space techniques and the complex method. In particular, we will use the complex interpretation of the Hardy classes of functions on the line with positive spectrum. In this section, we recall those basic facts of the theory of the Hardy classes $\hil^1(\R), \hil^2(\R)$ and $\hil^\infty(\R)$ which will be important in the coming sections. We discuss also properties of the shift operators $f(t) \mapsto e^{its}f(t)$ on weighted spaces on the real line, and their invariant subspaces.

\subsection{Hardy classes} \label{HardyClassSubsection}

For $p$ equal to $1$ or $2$, we denote by $\hil^p(\R)$ the subspace of $\L^p(\R, dx)$ consisting of those functions $f$ for which the Fourier transform $\widehat{f}$ vanishes on the negative part of the real axis: \[ \hil^p(\R) := \{ f \in \L^p(\R, dx) : \widehat{f}|\R_- \equiv 0\}. \]
It is a well-known fact that functions in the Hardy classes $\hil^1(\R)$ and $\hil^2(\R)$ admit a type of analytic extension to the upper half-plane \[ \H := \{ x + iy \in \mathbb{C} : y > 0 \}.\] We recall what exactly is meant by this extension and how it can be constructed. The Poisson kernel of the upper half-plane \[ \Po(t,x+iy) := \frac{1}{\pi}\frac{y}{(x-t)^2 + y^2}, \quad y > 0,\] admits a decomposition \begin{equation}\label{PoissonKernelDecomp} \Po(t,z) = \Re \Bigg(\frac{1}{\pi i(t-z)} \Bigg) = \frac{1}{ 2\pi i } \Bigg( \frac{1}{t-z} - \frac{1}{t-\conj{z}} \Bigg), \quad z = x+iy \in \H.\end{equation}
Since 
\begin{equation}
    \label{CauchyKernelFourierTransform}
    \frac{1}{t-\conj{z}} = -i \int_0^\infty e^{-i \conj{z} s} e^{its} \, ds
\end{equation} we may use Fubini's theorem to compute, in the case $f \in \hil^1(\R)$, that \begin{equation}
    \label{vanishingCauchyInt}\int_\R \frac{f(t)}{t-\conj{z}} \, dt = -i \int_0^\infty \Bigg(\int_\R f(t) e^{its} \, dt \Bigg) e^{-i\conj{z}s} ds = 0,
\end{equation} where the vanishing of the integral follows from \[ \int_\R f(t) e^{its} \, ds = \widehat{f}(-s) = 0, \quad s > 0,\] which holds for any $f \in \hil^1(\R)$ by the definition of the class. In the case $f \in \hil^2(\R)$ this argument does not work, but what instead works is an application of Plancherel's theorem and \thref{FourierTransformCauchyKernel} below to the first integral in \eqref{vanishingCauchyInt}, which again shows that this integral vanishes. Consequently, whenever $f \in \hil^p(\R)$ for $p=1, 2$, the formula \[ f(z) := \int_\R f(t) \Po(t,z) \, dt = \int_\R \frac{f(t)}{t-z} \frac{dt}{2\pi i}, \quad z \in \H\] defines, by the second integral expression above, an analytic extension of $f$ to $\H$. By the first expression, and classical properties of the Poisson kernel (see \cite[Chapter I]{garnett}), this extension satisfies 
\begin{equation} \label{pointwiseLimitExtensionEq}
    \lim_{y \to 0^+} f(x+iy) = f(x),\text{ for almost every } x \in \R.
\end{equation}
Moreover, we have \begin{equation}
    \label{H1property1}
\sup_{y > 0} \int_\R |f(x+iy)|^p \, dx  < \infty \end{equation} and 
\begin{equation}
    \label{H1property2}
    \lim_{y \to 0^+} \int_\R |f(x+iy) - f(x)|^p \, dx = 0.
\end{equation} if $f \in \hil^p(\R)$. The property \eqref{H1property1} follows readily from the Poisson integral formula for the extension of $f$ and Fubini's theorem. The property \eqref{H1property2} is a bit tricky to establish, and is proved in \cite[Chapter I, Theorem 3.1]{garnett}. In fact, the above listed properties characterize the functions in the Hardy classes.

\begin{prop}\thlabel{H1CharacterizationProp}
For $p = 1$ and $p=2$, a function $f \in \L^p(\R, dx)$ is a member of $\hil^p(\R)$ if and only if there exists an analytic extension of $f$ to $\H$ which satisfies the three properties in \eqref{pointwiseLimitExtensionEq}, \eqref{H1property1} and \eqref{H1property2}.
\end{prop}
The proposition is not hard to derive from \thref{H1FourierTransformFormula} below. Anyway, a careful proof can be found in \cite[p. 172]{havinbook}.

The following restriction on smallness of the modulus $|f|$ of a function $f \in \hil^1(\R)$ will be of crucial importance to us.

\begin{prop} \thlabel{HadyClassLogIntProp}
If $f \in \hil^1(\R)$, then \[ \int_\R \frac{\log |f(x)|}{1+x^2} dx > -\infty\] unless $f$ is the zero function.
\end{prop}

A proof of the proposition can be found in \cite[p. 35]{havinbook}. 

We shall also need to use the corresponding Hardy class of functions which are merely bounded on $\R$, and not necessarily integrable or square-integrable on $\R$. We use directly the complex interpretation of the class. Namely, we define $\hil^\infty(\R)$ to consist of those functions $f \in \L^\infty(\R, dx)$ which can be realized as limits \[ \lim_{y \to 0^+} f(x+iy) := f(x)\] for almost every $x \in \R$, where $f$ is bounded and analytic in $\H$. It can be checked that such $f$ has a distributional spectrum which vanishes on $\R_-$. Another important point is that if $f \in \hil^\infty(\R)$, then \[ \frac{f(x)}{(i+x)^2} \in \hil^1(\R),\] since we may apply \thref{H1CharacterizationProp} to the analytic function \[ z  \mapsto \frac{f(z)}{(i+z)^2}, \quad z \in \H.\] A function $h \in \hil^\infty(\R)$ of a given (bounded, measurable) modulus $|h| = W$ on $\R$ may be constructed by setting \begin{equation}
    \label{hOuterDef} \log h(z) :=  \frac{1}{\pi i }\int_\R \Big(\frac{1}{t-z} - \frac{t}{1+t^2} \Big)\log W(t) \, dt, \quad z \in \H,
\end{equation} and $h(z) := e^{\log h(z)}$. The integral above converges if \[\int_\R \frac{\log W(t)}{1+t^2} \, dt > -\infty\] which is a necessary condition for the construction to be possible. Then \[\log |h(z)| = \int \Po(t,z) \log W(t) \, dt,\] so that the equality $\lim_{y \to 0^+} |h(x+iy)| = |h(x)| = W(x)$ for almost every $x \in \R$ is a consequence of the well-known properties of the Poisson kernel. 

\subsection{A formula and an estimate for the Fourier transform of a Hardy class function} 

If $f \in \hil^1(\R)$, then the values of $\widehat{f}(\zeta)$ may be computed using a formula different from \eqref{FourierTransformDef}. To wit, denote by $f(z)$ the extension of $f$ to $\H$ which was discussed in Section \ref{HardyClassSubsection}. The function \[G_\zeta(z) := f(z) e^{-iz\zeta} = f(z)e^{-ix\zeta + y\zeta}, \quad z = x+iy \in \H\] is analytic in $\H$, and for this reason Cauchy's integral theorem implies 
\begin{equation}
\label{CauchyIntegralRectangle} 
\int_{R(\epsilon, y, a)} G_\zeta(z) dz = 0,
\end{equation} where $dz$ denotes the complex line integral, $\epsilon, y, a$ are all positive numbers, $\epsilon < y$, and $R(\epsilon, y, a)$ denotes the rectangular contour having as corners the four points with coordinates $(-a, \epsilon)$, $(a, \epsilon)$, $(a, y)$, $(-a, y)$, oriented counter-clockwise. 

Fix $y > 0$ and let $S_y$ denote the horizontal strip in $\H$ consisting of all complex numbers with imaginary part between $0$ and $y$. Then it follows from Fubini's theorem and \eqref{H1property1} that \begin{align*}
    \int_S |G_\zeta(z)| dA(z) \leq e^{y|\zeta|} \int_\R \int_0^y |f(x+is)|ds dx<\infty,
\end{align*} where $dA(z)$ denotes the area measure on the complex plane. This expresses the integrability on $\R$ of the continuous function \[ x \mapsto \int_0^y |G_\zeta(x+is)| ds. \] Hence there exists a positive sequence $\{a_n\}_n$ which satisfies \[ \lim_{n \to \infty} a_n = +\infty \] and for which \[ \lim_{n \to \infty} \int_0^y |G_\zeta(a_n + is)| ds + \int_0^y |G_\zeta(-a_n + is)| ds = 0.\] This means that \begin{align*}
    0 & = \lim_{n \to \infty} \int_{R(\epsilon, y, a_n)} G_\zeta(z) \, dz  \\ &= -\int_\R G_\zeta(x+iy) dx + \int_\R G_\zeta(x+ i\epsilon) dx
\end{align*} 
Moreover, equation \eqref{H1property2} quite easily implies \[ \lim_{\epsilon \to 0^+} \int_\R G_\zeta(x+i\epsilon) dx = \int_\R f(x)e^{-ix\zeta} dx = \sqrt{2\pi} \widehat{f}(\zeta).\] We have proven the following formula by combining the above two expressions.

\begin{prop}\thlabel{H1FourierTransformFormula}
For $f \in \hil^1(\R)$ we may compute the Fourier transform $\widehat{f}(\zeta)$ using the formula \[ \widehat{f}(\zeta) = e^{y\zeta} \int_\R f(x+iy)e^{-ix\zeta} \, d\lambda(x)\] for any choice of $y > 0$, where $f(x+iy)$ denotes the values of the analytic extension of $f$ to $\H$.
\end{prop}

This formula has the following simple corollary which will be of critical importance below.

\begin{cor} \thlabel{FourierDecayFromExtensionGrowthCorollary}
    If $h \in \hil^\infty(\R)$ has an analytic extension to $\H$ which satisfies, for some constant $c > 0$, an estimate of the form \[ \sup_{x \in \R} \, |h(x+iy)| \leq e^{c/y}, \quad \text{ for all } y > 0,\] then the Fourier transform $\widehat{h_*}$ of the function \[ h_*(x) := \frac{h(x)}{(i+x)^2} \in \hil^1(\R)\] satisfies \[ |\widehat{h_*}(\zeta)| \leq \sqrt{\frac{\pi}{2}} e^{2 \sqrt{c}\sqrt{\zeta}}, \quad \zeta > 0.\]
\end{cor}

\begin{proof}
    It was mentioned in Section \ref{HardyClassSubsection} that $h_* \in \hil^1(\R)$. Therefore, we may use \thref{H1FourierTransformFormula} to estimate 
    \begin{align*}
        |\widehat{h_*}(\zeta)| &\leq e^{y\zeta}\int_\R \frac{|h(x+iy)|}{|i + x + iy|^2}\, d\lambda(x) \\ &\leq e^{y\zeta}\int_\R \frac{e^{c/y}}{1+x^2} \, d\lambda(x)
        \\ &= \sqrt{\frac{\pi}{2}} e^{y\zeta + c/y}.
    \end{align*} Since $y > 0$ can be freely chosen, we may now set it to $y = \sqrt{c/\zeta}$ to obtain the desired estimate.
\end{proof}

\subsection{A semigroup of operators and its invariant subspaces} If $w \in \L^1(\R, dx)$ and $s \in \R$, the operator $U^s: \L^2(\R, w \, dx) \to \L^2(\R, w \, dx)$ given by \[U^sf(x) := e^{isx}f(x) \] is unitary on $\L^2(\R, w \, dx)$. We shall be interested in subspaces of $\L^2(\R, w\, dx)$ which are invariant for the operators in the semigroup $\{U^s\}_{s > 0}$. Given any element $f \in \L^2(\R, w\, dx)$, we denote by $[f]_w$ the smallest closed linear subspace of $\L^2(\R, w \, dx)$ which contains $f$ and also all the functions $U^sf$, $s > 0$. 

\begin{prop} \thlabel{UsInvSubspacesProp} Let $f \in \L^2(\R, w \, dx)$ be a non-zero element which satisfies \[\int_\R \frac{\log \big(|f(x)|^2 w(x) \big) }{1+x^2} dx = -\infty.\] Then the subspace $[f]_w$ coincides with $L^2(E, w \, dx)$, where $E = \{ x \in \R : |f(x)| > 0 \}$.
\end{prop}

\begin{remark} As usual, the set $E$ above is defined in a bit imprecise way. Since $f$ is, strictly speaking, merely a representative of an equivalence class of measurable functions in $\L^2(\R, w \, dx)$, the set $E$ is not well-defined pointwise. However, it is well-defined up to a set of Lebesgue measure zero, and so the initial choice of the representative is unimportant.
\end{remark}

\begin{proof}[Proof of \thref{UsInvSubspacesProp}]
Since the function $f$ vanishes almost everywhere outside of the set $E$, then so does $U^sf$ for any $s > 0$. Consequently, $[f]_w \subset \L^2(E, w \, dx)$. Conversely, let us consider an element $g \in \L^2(E, w \, dx)$ with the property that \[ \int_\R U_sf(x) \conj{g(x)} w(x) dx = \int_\R e^{isx} f(x) \conj{g(x)} w(x) dx = 0, \quad s > 0.\] Setting $h := f \conj{g} w \in \L^1(\R, dx)$, we note that the vanishing of the integrals above is equivalent to $h$ being a member of the Hardy class $\hil^1(\R)$. We note also that \[ \int_\R \frac{\log |h(x)|}{1+x^2} dx = \frac{1}{2}\int_\R \frac{\log \big( |f(x)|^2 w(x) \big)}{1+x^2} dx + \frac{1}{2}\int_\R \frac{\log \big( |g(x)|^2 w(x) \big)}{1+x^2} dx.\] The above equality is to be interpreted in a generalized sense: the first integral on the right-hand side is divergent by our assumption, and so may the second, but their positive parts are certainly finite by the assumption that $f,g \in \L^2(\R, w \, dx)$. This implies that \[ \int_\R \frac{\log |h(x)|}{1+x^2} = -\infty.\] \thref{HadyClassLogIntProp} now shows that $h = f \conj{g}w$ must be the zero function. Since $|f(x)|w(x) > 0$ on $E$ and $g$ vanishes outside of $E$, this means that $g \equiv 0$. So $[f]_w$ is a closed and dense subspace of $\L^2(E, w\,dx)$, which means that the two spaces are equal.     
\end{proof}

\begin{cor}\thlabel{invSubspaceCorollary}
    If $f \in \L^2(\R, w\, dx)$ is also a member of $\L^p(\R, w \, dx)$ for some $p > 2$, and if \[\int_\R \frac{\log w(x)}{1+x^2} \, dx = -\infty, \] then $[f]_w$ coincides with $\L^2(E, w \, dx)$, where $E = \{x \in \R : |f(x)| > 0 \}$.
\end{cor}

\begin{proof}
    To prove the corollary we need to verify the condition in \thref{UsInvSubspacesProp}. Note that, pointwise, we have \[ \log\big( |f|^2 w\big) = (2/p) \log \big( |f|^p w\big) + (1-2/p) \log w.\] The coefficients $2/p$ and $1-2/p$ are positive. The inequality $\log (x) \leq x$ for $x > 0$ shows that \[ \int_\R \frac{\log \big( |f(x)|^p w(x) \big) }{1+x^2} \, dx \leq \int_\R |f(x)|^p w(x) dx < +\infty.\] Note that the integral on the left might very well be equal to $-\infty$, but that is of no concern to us: we conclude from the assumption, and the pointwise inequality above, that \[ \int_\R \frac{\log \big( |f(x)|^2 w(x)\big)}{1+x^2} \,dx = -\infty \] and apply \thref{UsInvSubspacesProp}.
\end{proof}

\section{A product space and its Hardy subspace}
\label{ProductSpaceSection}

Let $\rho: \R_+ \to \R_+$ be a bounded, continuous, non-negative and decreasing function, and $w \in \L^1(\R, dx) \cap \L^\infty(\R, dx)$ be a non-negative function. We consider the product space $\L^2(\R, w\, dx) \oplus \L^2(\R_+, \rho \, dx)$. Inside of this space we embed the linear manifold $\hil^1(\R) \cap \hil^2(\R)$ in the following way:
\[ Jf := (f, \widehat{f}) \in \L^2(\R, w\, dx) \oplus \L^2(\R_+, \rho \, dx), \quad f \in \hil^1(\R) \cap \hil^2(\R). \] The tuple $Jf$ is well-defined as an element of the product space, since both $f$ and $\widehat{f}$ are members of $\L^2(\R, dx)$ and both $\rho$ and $w$ are bounded. We define the \textit{Hardy subspace} $\hil(w,\rho)$ as the norm-closure of the linear manifold \[ \{ Jf : f \in \hil^1(\R) \cap \hil^2(\R) \}\] inside of the product space $\L^2(\R, w\, dx) \oplus \L^2(\R_+, \rho \, dx)$. Thus each tuple $(h,k) \in \hil(w, \rho)$ has the property that there exists some sequence $\{f_n\}_{n}$ of functions in $\hil^1(\R) \cap \hil^2(\R)$ such that \[h = \lim_{n \to \infty} f_n\] in the space $\L^2(\R, w \, dx)$, and simultaneously \[ k = \lim_{n \to \infty} \widehat{f_n}\] in the space $\L^2(\R_+, \rho \, dx)$. 

We could have used a set of tuples $Jf$ with $f \in \hil^2(\R)$ in the definition of the Hardy subspace, and arrived at the same space. Indeed, we have the following proposition.

\begin{prop}\thlabel{KernelContainmentHardySubspace}
    With $w$ and $\rho$ as above, the Hardy subspace $\hil(w, \rho)$ contains all tuples of the form $(f, \widehat{f})$, $f \in \hil^2(\R)$. Moreover, tuples $Jf$ where $f \in \hil^1(\R)\cap\hil^2(\R)$ and $f$ extends analytically to a half-space $\{ z = x+iy \in \mathbb{C} : y > -\delta \}$, $\delta = \delta(f) > 0$, are norm-dense in $\hil(w, \rho)$.
\end{prop}

\begin{proof}
    Fix $f \in H^2(\R)$, and consider the functions $f_\epsilon(x)$ defined by the formula
    \[f_\epsilon(x) :=  \frac{i f(x+i\epsilon)}{\epsilon x + i}, \quad x \in \R, \epsilon > 0.\] These functions are contained in $\hil^1(\R) \cap \hil^2(\R)$ for each $\epsilon > 0$, and they are analytic in a half-space larger than $\H$. Note that \[ \Bigg\vert \frac{i}{\epsilon x + i} \Bigg\vert \leq 1, \quad x \in \R\] and that \[ \lim_{\epsilon \to 0^+} \frac{i}{\epsilon x + i} = 1.\] We readily see from \thref{H1CharacterizationProp} and the dominated convergence theorem that we have \[ \lim_{\epsilon \to 0^+} \int_\R |f_\epsilon - f|^2 dx = 0.\] By Plancherel's theorem, we therefore also have \[ \lim_{\epsilon \to 0^+} \int_{\R_+} |\widehat{f_\epsilon} - \widehat{f}|^2 d\zeta = 0.\] A fortiori, we have \[ \lim_{\epsilon \to 0^+} \int_\R |f_\epsilon - f|^2 w\, dx = 0\] and \[ \lim_{\epsilon \to 0^+} \int_{\R_+} |\widehat{f_\epsilon} - \widehat{f}|^2 \rho \, d\zeta = 0.\] Thus, as $\epsilon \to 0$, the tuples $Jf_\epsilon \in \hil(w, \rho)$ converge in the norm of the space to the tuple $(f, \widehat{f})$, which is therefore contained in $\hil(w, \rho)$. This proves the first statement of the proposition. The second has the same proof, we merely start with $f \in \hil^1(\R) \cap \hil^2(\R)$ and run the same argument.
\end{proof}

The shift operators \[U^s f(x) = e^{ixs}f(x), \quad f \in \L^2(\R, w \, dx)\] are unitary. Using the convention that $g(x) \equiv 0$ for $x < 0$ and $g \in \L^2(\R_+, \rho \, dx)$, the translation operators \[\widehat{U^s} g(x) := g(x - s), \quad g \in \L^2(\R_+, \rho \, dx)\] are contractions on $\L^2(\R_+, \rho \, dx)$, whenever $s > 0$. This fact is a consequence of the assumption that $\rho$ is decreasing:

\begin{align*}
    \int_{\R_+} |g(x - s)|^2 \rho(x) dx &= \int_s^{\infty} |g(x-s)|^2 \rho(x) dx \\
    & \leq \int_s^\infty |g(x-s)|^2 \rho(x-s) dx \\
    & = \int_{\R_+} |g(x)|^2 \rho(x) dx.
\end{align*}
We used that $g(x-s)$ vanishes for $x \in (0,s)$. Consequently, the operators \[U^s_* := U^s \oplus \widehat{U^s}, \quad s > 0\] are bounded on the space $\L^2(\R, w\, dx) \oplus \L^2(\R_+, \rho \, dx)$. Moreover, the Hardy subspace $\hil(w, \rho)$ is invariant for these operators. Indeed, if $f \in \hil^1(\R) \cap \hil^2(\R)$, then by the well-known property of the Fourier transform $\widehat{U^s}\widehat{f} = \widehat{U^s f}$, we obtain \[ U_*^sJf = (U^sf, \widehat{U^s} \widehat{f}) = (U^sf, \widehat{U^s f}) = JU^sf.\] The function $U^s f$ is contained in $\hil^1(\R) \cap \hil^2(\R)$, and so the above relation shows that a dense subset of $\hil(w, \rho)$ is mapped into $\hil(w, \rho)$ under each of the bounded operators $U^s_*$. The mentioned invariance follows.

\section{Strategy of the proofs}
\label{StrategySection}

This section outlines the strategy of the proofs of \thref{CondensationTheorem} and \thref{SparsenessTheorem}. 

\subsection{Two easy computations}

In our strategy, we will need to use the results of the following two computations.

\begin{lem} \thlabel{FourierTransformCauchyKernel}
Let \[ \psi_z(x) := \frac{i}{\sqrt{2\pi}(x-\conj{z})}, \quad z \in \H.\]
The Fourier transform $\widehat{\psi_z}$ equals \[ \widehat{\psi_z}(\zeta) = e^{-i\conj{z}\zeta} \mathbbm{1}_{\R_+}(\zeta),\] and the Fourier transform $\widehat{\conj{\psi_z}}$ of the conjugate of $\psi_z$ equals \[\widehat{\conj{\psi_z}}(\zeta) = e^{-i z \zeta} \mathbbm{1}_{\R_-}(\zeta).\]
\end{lem}

\begin{proof}
    It is perhaps easiest to apply the Fourier inversion formula to the asserted formula for $\widehat{\psi_z}$. We readily compute
    \begin{align*}
        \int_\R \widehat{\psi_z}(\zeta)e^{i\zeta x} \,d\lambda(\zeta) &= \frac{1}{\sqrt{2\pi}} \int_0^\infty e^{(-i\conj{z} + ix)\zeta} \, d\zeta \\ & = \psi_z(x).
    \end{align*}
    The other formula follows from $\widehat{\conj{\psi_z}}(\zeta) = \conj{\widehat{\psi_z}(-\zeta)}$, which is an easily established property of the Fourier transform.
\end{proof}

\begin{prop} \thlabel{ConvolutionFourierDecayLemma}
    Assume that $f \in \L^2(\R, dx)$ satisfies \[\rho_{\widehat{f}}(x) = \int_x^\infty |\widehat{f}(\zeta)|  d\zeta = \mathcal{O}\big(e^{-c\sqrt{x}}\big), \quad x > 0\] for some $c > 0$, and let \[s(x) := \overline{\psi_i(x)} = \frac{-i}{\sqrt{2 \pi} (x-i)}.\] Then \[ \big\vert\widehat{f s}(\zeta)\big\vert = \mathcal{O}\big( e^{-c\sqrt{\zeta}} \big), \quad \zeta > 0.\]
\end{prop}

\begin{proof}
    Note that $f s \in \L^1(\R, dx)$, and recall that the Fourier transform $\widehat{f s}$ is thus a continuous function given by the convolution of the Fourier transforms of $f$ and $s$. By \thref{FourierTransformCauchyKernel}, we obtain \[\widehat{s}(\zeta) = e^{\zeta}\mathbbm{1}_{\R_-}(\zeta),\] and so \[\widehat{f s}(\zeta) = \int_{\R} \widehat{f}(x)e^{\zeta - x}\mathbbm{1}_{\R_-}(\zeta - x)  \, d\lambda(x) = \int_\zeta^\infty \widehat{f}(x) e^{\zeta - x} \, d\lambda(x).\] The exponential term in the last integral is bounded by $1$. Therefore $\big\vert\widehat{fs}(\zeta)\big\vert \leq \rho_{\widehat{f}}(\zeta)$, and the desired estimate follows from the decay assumption on $\rho_{\widehat{f}}$.
\end{proof}

\subsection{Strategy of the proof of \thref{CondensationTheorem}} Given a function $f \in \L^2(\R, dx)$ we consider the set \[E := \{ x \in \R : |f(x)| > 0 \},\] which is well-defined up to a set of Lebesgue measure zero. Let $\mathcal{F}$ denote the family of all finite open intervals $I$ which satisfy \[ \int_I \log |f(x)| \, dx > -\infty,\] and set \[U := \cup_{I \in \mathcal{F}} I.\] Since $\log |f| \equiv -\infty$ on $I \setminus E$ for every interval $I$, it follows that if $\log |f|$ is integrable on $I$, then the set difference $I \setminus E$ must have measure zero. Consequently, since one can easily argue that we can express $U$ as a \textit{countable} union of intervals $I$ on which $\log |f|$ is integrable, the Lebesgue measure of the set difference $U \setminus E$ must be zero. However, the set difference $E \setminus U$ might have positive measure. We set \[\res(f) := E \setminus U\] and call this set the \textit{residual} of $f$. The residual is well-defined up to a set of Lebesgue measure zero.

\begin{claim*} \thlabel{claim1} Under the assumption that $\rho_{\widehat{f}}(\zeta) = \mathcal{O}\big( e^{-c\sqrt{\zeta}}\big)$ for some $c > 0$, the set $\res(f)$ has Lebesgue measure zero.
\end{claim*}

\thref{CondensationTheorem} follows immediately from the above claim. Indeed, the roles of $f$ and $\widehat{f}$ may obviously be interchanged in the statement of \thref{CondensationTheorem}, and the above claim implies that the open set $U$ equals $E$ up to an error of measure zero. Local integrability of $\log |f|$ on the set $U$ follows from its construction. 

We set \[w(x) := \min( |f(x)|^2, 1).\] Note that $\res(w) = \res(f)$ and that $w \in \L^1(\R, dx)$. Our \thref{claim1} will follow from the next assertion. 

\begin{claim*} \thlabel{claim2} Let $\rho: \R_+ \to \R_+$ be a bounded, continuous, non-negative and decreasing function which satisfies $\rho(x) = \mathcal{O}\big( e^{-d\sqrt{x}}\big)$ for some $d > 0$ and $x > 0$. Then, every tuple of the form \[(h, 0) \in \L^2(\R, w \, dx) \oplus \L^2(\R_+, \rho \, dx),\] where $h$ is any function in $\L^2(\R, dx)$ which lives only on the set $\res(w)$, is contained in the Hardy subspace $\hil(w,\rho)$.
\end{claim*}

To prove \thref{claim1} from \thref{claim2} we will use a trick involving Plancherel's theorem. We set $\rho(x) = e^{-c \sqrt{x}}$, where $c > 0$ is the constant appearing in \thref{claim1}. Let $h$ be as in \thref{claim2}, and \[ s(x) := \overline{\psi_i}(x) = \frac{-i}{\sqrt{2 \pi} (x-i)}\] be as in \thref{ConvolutionFourierDecayLemma}. We will show that \[ \int_\R h \conj{f s} \, dx = 0.\]
This implies, by the generality of $h$, that $fs$ is zero on the set $\res(w) = \res(f)$. Since $s$ is non-zero everywhere on $\R$, in fact $f$ is zero on $\res(f)$. Since $\res(f) \subset E = \{ x \in \R : |f(x)| > 0\}$, it follows that the residual has Lebesgue measure zero. Thus establishing the vanishing of the above integral is sufficient to prove \thref{claim1} from \thref{claim2}. We do so next.

Because $(h,0) \in \hil(w, \rho)$, there exists a sequence $\{g_n\}$ of functions $g_n \in \hil^1(\R) \cap \hil^2(\R)$ such that $g_n \to h$ in the norm of $\L^2(\R, w\, dx)$ and $\widehat{g_n} \to 0$ in the norm of $\L^2(\R_+, \rho \, dx)$. Consider the quantities \[\int_\R (h-g_n)\conj{f s} \, d\lambda = \int_E (h-g_n)\sqrt{w} \frac{\conj{fs}}{\sqrt{w}} d\lambda.\] We passed from domain of integration $\R$ into $E$, since $f$ vanishes outside of $E$ anyway (note also that $w > 0$ almost everywhere on $E$). By the Cauchy-Schwarz inequality, we obtain 

\begin{align*}
    \Bigg\vert \int_\R (h-g_n)\conj{f s} \, d\lambda \Bigg\vert \leq \sqrt{\int_\R |h-g_n|^2 w \,d\lambda} \sqrt{\int_E \frac{|f|^2|}{w}|s|^2 d\lambda}.
\end{align*} Note that the first of the factors on the right-hand side of the inequality above converges to $0$. The other factor is finite. Indeed, since $|f|^2/w \equiv 1$ on the set where $|f| < 1$, and $|f|^2/w = |f|^2$ on the set where $|f| \geq 1$, we obtain $\frac{|f|^2}{w} \leq |f|^2 + 1$, and consequently \[ \frac{|f|^2}{w}|s|^2  \leq |f|^2|s|^2 + |s|^2 \in \L^1(\R, dx).\]
This computation implies the formula \[ \int_\R h \conj{f s} \,d\lambda = \lim_{n \to \infty} \int_\R g_n \conj{f s} \, d\lambda = \lim_{n \to \infty}\int_{\R_+} \widehat{g_n} \conj{\widehat{f s}} \, d\lambda,\] where we used Plancherel's theorem in the last step. Now, recall that by \thref{ConvolutionFourierDecayLemma} we may estimate
\[ \Bigg\vert \int_{\R_+} \widehat{g_n} \conj{\widehat{f s}} \, d\lambda \Bigg\vert \leq A \int_{\R_+} |\widehat{g_n}(\zeta)| e^{-c \sqrt{\zeta}} \, d\lambda(\zeta) \] for some positive constant $A$. By again using Cauchy-Schwarz inequality, we obtain 
\begin{align*}
    \Bigg\vert \int_\R \widehat{g_n} \conj{\widehat{f s}} \, d\lambda \Bigg\vert \leq A \sqrt{ \int_{\R_+} |\widehat{g_n}(\zeta)|^2 e^{-c \sqrt{\zeta}} \, d\lambda(\zeta)} \sqrt{\int_{\R_+} e^{-c \sqrt{\zeta}} \, d\lambda(\zeta)}.
\end{align*} The second factor on the right-hand side is certainly finite. Since $\rho(x) = e^{-c\sqrt{x}}$ and $g_n \to 0$ in $\L^2(\R_+, \rho\, dx)$, the first factor above converges to $0$, as $n \to \infty$. All in all, we have obtained that \[ \int_\R h \conj{f s} \, d\lambda = \lim_{n \to \infty} \int_{\R_+} \widehat{g_n} \conj{\widehat{f s}} \, d\lambda = 0.\] By the earlier discussion, this is sufficient to establish \thref{claim1} from \thref{claim2}. We need to prove \thref{claim2} in order to prove \thref{CondensationTheorem}. We will do so in the coming sections.

\subsection{Strategy of the proof of \thref{SparsenessTheorem}}

We will derive \thref{SparsenessTheorem} from the following claim.

\begin{claim*} \thlabel{claim3} There exists a compact set $E \subset \R$ of positive Lebesgue measure, and an increasing function $M: \R_+ \to \R_+$ which satisfies \begin{equation} \label{MfastIncrease}
\lim_{x \to \infty} \frac{M(x)}{x^a} = \infty
\end{equation} for every $a \in (0, 1/2)$, such that if \[\rho(x) := e^{-M(x)}, \quad x > 0, \] then the Hardy subspace $\hil(\mathbbm{1}_E, \rho)$ is properly contained in $\L^2(E, \mathbbm{1}_E \, dx) \oplus \L^2(\R_+, \rho \, dx)$.
\end{claim*}

We proceed to show how one proves \thref{SparsenessTheorem} from this claim. Let $E$ and $\rho = e^{-M}$ be as in \thref{claim3}, and assume that the non-zero tuple $(h,k) \in \L^2(E, \mathbbm{1}_E \, dx) \oplus \L^2(\R_+, \rho \, dx)$ is orthogonal to $\hil(\mathbbm{1}_E, \rho)$. We shall soon see that $h$ in fact is non-zero. The function $h$ lives only on the set $E$, and we will show that it has the required spectral decay. Recall \thref{FourierTransformCauchyKernel}, let $\psi_z \in \hil^2(\R)$ be as in that lemma, and let $g$ be the inverse Fourier transform of $k\rho$, so that $\widehat{g} = k\rho \in \L^2(\R_+, dx)$. In fact $g \in \hil^2(\R)$, since its spectrum is positive.  The orthogonality means that
\begin{align*}
0 &= \int_E h\conj{\psi_z} \, d\lambda  + \int_{\R_+} k\conj{\widehat{\psi_z}} \rho \,d\lambda \\ & = \int_E h \conj{\psi_z}  \, d\lambda + \int_{\R} g \conj{\psi_z} \, d\lambda.
\end{align*} We used Plancherel's theorem. The above relation shows that the function $G := h + g \in \L^2(\R, dx)$ is orthogonal in $\L^2(\R, dx)$ to each of the functions $\psi_z$. Let $P: \L^2(\R, dx) \to \hil^2(\R)$ be the orthogonal projection. In terms of Fourier transforms, we have \[ \widehat{Ph}(\zeta) = \widehat{h}(\zeta) \mathbbm{1}_{\R_+}(\zeta).\] Then \[ G_0 := PG = Ph + g\] is orthogonal not only to $\psi_z$, but also to $\conj{\psi_z}$, since by \thref{FourierTransformCauchyKernel} the functions $\conj{\psi_z}$ have spectrum contained in $\R_-$. But then the decomposition formula for the Poisson kernel in \eqref{PoissonKernelDecomp} shows that \[ \int_\R G_0(t) \Po(t,z) \, dt = 0\] for each $z \in \H$, and it is an elementary fact about the Poisson kernel that we must, in this case, have $G_0 \equiv 0$. So $Ph = -g$. We can now argue that $h \neq 0$. Indeed, if $h = 0$, then $g = 0$. Since $\widehat{g} = k\rho$, that would mean $k = 0$, contradicting that the tuple $(h,k)$ is non-zero. Having established that $h \neq 0$, we proceed by taking Fourier transforms to obtain \[\widehat{h} \mathbbm{1}_{\R_+} = \widehat{Ph} = -\widehat{g} = -k \rho. \] Using Cauchy-Schwarz inequality, we may now estimate
\begin{align*}
    \rho_{\widehat{h}}(x) & = \int_x^\infty |k| \rho \, d\zeta \\
    &\leq \sqrt{ \int_x^\infty \rho \,d\zeta}  \sqrt{ \int_x^\infty |k|^2 \rho \,d\zeta}, \quad x > 0.
\end{align*} The second factor is finite, and the growth of $M$ asserted in \thref{claim3} implies that for every fixed $a \in (0,1/2)$ there exists a constant $C(a) > 0$ for which we have  \[\rho(\zeta) = e^{-M(\zeta)} \leq e^{-C(a)\zeta^a}, \quad \zeta > 0.\] It follows that the integral inside the square root of the first factor above satisfies \[ \int_x^\infty \rho \,d\zeta \leq \int_x^\infty e^{-C(a)\zeta^a} \, d\zeta = \mathcal{O}\big( e^{-C(a)x^a} \big), \quad x > 0.\] Since $a \in (0, 1/2)$ was arbitrary, we conclude that $\rho_{\widehat{h}}(x) = \mathcal{O}\big( e^{-x^a})$ for every $a \in (0,1/2)$ and $x > 0$. This easily implies \thref{SparsenessTheorem}. 

It follows that \thref{SparsenessTheorem} is implied by \thref{claim3}.

\section{Proof of \thref{CondensationTheorem}}
Our proof is an adaptation to the half-plane setting of the authors' technique from \cite{malman2023revisiting}, and in fact the two proofs are very similar. The problem studied in \cite{malman2023revisiting} is different, but in both the present work and in the reference, the main trick consists of constructing a highly oscillating sequence of functions which simultaneously obey appropriate spectral bounds.

\subsection{A sufficient construction} \label{constructSubsectionDiscussion} We start by reducing our task to a construction of a certain sequence of bounded functions. Recall that $w(x) = \min( |f(x)|^2 ,1) \in \L^1(\R, dx)$ and that $\rho$ has the decay $\rho(\zeta) = \mathcal{O}\big( e^{-d\sqrt{\zeta}}\big)$ for $\zeta > 0$ and some $d > 0$. Note that we may assume throughout that \[\int_\R \frac{\log w(x)}{1+x^2} \, dx = -\infty.\] Indeed, if on the contrary this integral converges, then $\res(w) = \res(f)$ is void, and both \thref{claim2} of Section \ref{StrategySection} and \thref{CondensationTheorem} (with $f$ and $\widehat{f}$ playing opposite roles) hold trivially.

We may decompose $\res(w)$ as \[ \res(w) = \cup_{m \geq 1} \, F_m\] where \begin{equation}
    \label{resWdecompFm} F_m :=  [-m, m] \cap \{x \in \R : w(x) > 1/m \} \cap \res(w).
\end{equation} The set equality above holds up to an error of measure zero. The sets $F_m$ are bounded, and on each of them $w$ is bounded from below.

\begin{prop} \thlabel{hnSplittingSequence}
    In order to establish \thref{claim2}, it suffices to construct, for any fixed $m \in \mathbb{N}$ and $c > 0$, a sequence of functions $\{h_n\}_n$ in $\hil^\infty(\R)$ which has the following properties.

    \begin{enumerate}[(i)]
        \item The analytic extensions of the functions $h_n$ to $\H$ satisfy the bound  $|h_n(x+iy)| \leq e^{\frac{c}{y}}$ for $y > 0$,
        \item  $\lim_{n \to \infty} h_n(x) = 0$ for almost every $x \in F_m$,
        \item $\lim_{n \to \infty} h_n(z) = 1$ for every $z \in \H$,
        \item there exists an $A > 0$ and $p > 2$ such that $|h(x)|^p w(x) < A$ for almost every $x \in \R$.
    \end{enumerate}
\end{prop}

\begin{proof} Property $(i)$, together with \thref{FourierDecayFromExtensionGrowthCorollary}, implies that the functions $g_n(x) = \frac{h_n(x)}{(i+x)^2} \in \hil^1(\R) \cap \hil^2(\R)$ obey the spectral bound $|\widehat{g_n}(\zeta)| \leq \sqrt{\frac{\pi}{2}} e^{2\sqrt{c}\sqrt{\zeta}}, \zeta > 0$. Since $\rho(\zeta) = \mathcal{O}\big(e^{-d\sqrt{\zeta}}\big)$ for some $d > 0$, the spectral bound implies that \[ \sup_n \int_{\R_+} |\widehat{g_n}|^2 \rho \, d\zeta < \infty \] if $c$ is small enough. Together with $(iv)$, we see that $Jg_n = (g_n, \widehat{g_n})$ forms a bounded subset of the Hardy subspace of product space $\L^2(\R, w \, dx) \oplus \L^2(\R_+, \rho \, dx)$. We may thus assume, by passing to a subsequence, that $\{Jg\}_n$ tends weakly in $\hil(w, \rho)$ to some tuple $(h,k)$. In fact, $(iv)$ implies that $\{g_n\}_n$ is a sequence bounded in $\L^p(\R, w \, dx)$, so we have that $h \in \L^p(\R, w \, dx)$ for some $p > 2$. The fact that $h \equiv 0$ on $F_m$ is a consequence of the weak convergence of $g_n$ to $h$ and the condition $(ii)$, which implies $\lim_{n \to \infty} g_n(x) = 0$ for almost every $x \in F_m$. Moreover, by the formula in \thref{H1FourierTransformFormula}, we have \[ \widehat{g_n}(\zeta) = e^{y\zeta} \int_\R \frac{h_n(x+iy)}{(i+x+iy)^2} e^{-ix\zeta} d\lambda(x)\] for any $y > 0$. The integrand converges pointwise to $\frac{e^{-ix\zeta}}{(i+x+iy)^2}$ by $(iii)$, and it is dominated pointwise by the integrable function \[ x  \mapsto \frac{e^{\frac{c}{y}}}{1+x^2}, \quad x \in \R.\] The dominated convergence theorem implies that \[ \lim_{n\to\infty} \widehat{g_n}(\zeta) = e^{y\zeta} \int_\R \frac{1}{i+x+iy} e^{-ix\zeta} \, d\lambda(x) = \int_\R \frac{1}{(i+x)^2} e^{-ix\zeta} \, d\lambda(x) = \widehat{\Psi}(\zeta), \] where $\Psi(x) := \frac{1}{(i+x)^2} \in \hil^1(\R) \cap \hil^2(\R)$. In the next-to-last equality we used \thref{H1FourierTransformFormula} backwards. Weak and pointwise convergence implies, as previously, that $k = \widehat{\Psi}$. Since $J\Psi \in \hil(w, \rho)$, we have that $(h,k) - J\Psi = (h-\Psi, 0) \in \hil(w, \rho)$. The function $h - \Psi$ is non-zero almost everywhere on $F_m$. Indeed, $h$ vanishes on $F_m$, and $\Psi(x)$ is non-zero everywhere on $F_m$. Also, since $h \in \L^p(\R, w\, dx)$, we have $h - \Psi \in \L^p(\R, w \, dx)$. The conditions to apply \thref{invSubspaceCorollary} are thus satisfied, and by the invariance of $\hil(w, \rho)$ under the operators $U^*_s$ defined in Section \ref{ProductSpaceSection}, we conclude that \[ \L^2(E, w\, dx) \oplus \{0\} \subset \hil(w, \rho),\] where $F_m \subset E := \{ x \in \R : |h(x) - \Psi(x)| > 0 \}$. Since $m$ is arbitrary, we conclude that \[\L^2( \cup_m F_m, w\, dx) \oplus \{ 0\} = \L^2( \res(w), w\, dx) \oplus \{ 0\} \subset \hil(w, \rho). \] This is sufficient to conclude the validity of \thref{claim2}.
\end{proof}

\subsection{An estimate for Poisson integrals}

Let $\mu$ be a finite real-valued measure on $\R$. The \textit{Poisson integral} of $\mu$ is the harmonic function $\Po_\mu : \H \to \R$ which is given by the formula \[ \Po_\mu(z) := \int_\R \Po(t,z) \, d\mu(t) = \frac{1}{\pi}\int_\R \frac{y}{(x-t)^2 + y^2} \, d\mu(t), \quad z = x+iy \in \H.\] By the triangle inequality, and an estimation of $\Po(t,z)$ by its supremum $\frac{1}{\pi y}$ over $\R$, we easily obtain the inequality \[ \big\vert \Po_\mu(z) \big\vert \leq \frac{|\mu|(\R)}{\pi y}, \quad z = x+iy \in \H,\] and where $|\mu|$ denotes the usual variation of the measure $\mu$. We obtain a much better inequality for measures $\mu$ which are oscillating rapidly. The following lemma is the half-plane version of an estimate in \cite[Lemma 3.2]{malman2023revisiting}.

\begin{lem} \thlabel{PsnIntegralEstimate}
Let $\mu$ be a finite real-valued measure on $\R$ which has the following structure: there exists a finite sequence of disjoint intervals $\{I_j\}_j$ of $\R$, and a decomposition $\mu = \sum_j \mu_j$, where $\mu_j$ is a real-valued measure supported inside $I_j$, $\mu_j(I_j) = 0$, and $|\mu_j|(I_j) \leq C$ for some $C > 0$ which is independent of $j$. Then \[ \big\vert\Po_\mu(x+iy)\big\vert \leq \frac{C}{\pi y}, \quad x+iy \in \H.\]    
\end{lem}

\begin{proof}
    Since $-\mu$ satisfies the same conditions as $\mu$, and so does any translation of $\mu$, it suffices to prove that $\Po_\mu(y) \leq \frac{C}{\pi y}$ for any $y > 0$. We have \[\Po_\mu(y) = \sum_j \frac{1}{\pi}\int_\R \frac{y}{t^2 + y^2} d\mu_j(t). \] If $\mu_j = \mu_j^+ - \mu_j^-$ is the decomposition of $\mu_j$ into its positive and negative parts, then we have the estimate \begin{align*}
        \Po_\mu(y) &\leq \frac{1}{\pi}\sum_j \sup_{t \in I_j} \frac{y}{t^2 + y^2} \cdot \mu^+_j(I_j) - \inf_{t \in I_j} \frac{y}{t^2 + y^2} \cdot \mu^-_j(I_j) \\ & \leq \frac{C}{2\pi }\sum_{j} \sup_{t \in I_j} \frac{y}{t^2 + y^2} - \inf_{t \in I_j} \frac{y}{t^2 + y^2} \\
        & := \frac{C}{2\pi} S
    \end{align*} In the second step we used that $\mu_j(I_j) = \mu_j^+(I_j) - \mu_j^-(I_j) = 0$ and that $|\mu_j|(I_j) = \mu_j^+(I_j) + \mu_j^-(I_j) \leq C$. Since the intervals $\{I_j\}_j$ are disjoint, and the function $t \mapsto \frac{y}{t^2 + y^2}$ is increasing for $t < 0$, decreasing for $t > 0$, and attains a maximum value of $1/y$ at $t = 0$, the sum $S$ in the above estimate cannot be larger than $\frac{2}{y}$ (it is easily seen to be bounded by twice the height of the graph of $t \mapsto \frac{y}{t^2 + y^2}$). The estimate follows. 
\end{proof}

\subsection{The construction}
In accordance with the earlier discussion in Section \ref{constructSubsectionDiscussion}, we recall the decomposition \eqref{resWdecompFm} and assume below that $F := F_m$ is a bounded subset of $\res(w) \cap \{ x \in \R : w(x) > \delta\}$ for some $\delta > 0$. The set $F$ inherits the following property from $\res(w)$: if $I$ is an interval, and $|I \cap F| > 0$, then $\int_I \log w \, dx = -\infty.$ 

\begin{lem} \thlabel{pieceLemma}
If $|I \cap F| > 0$ for some finite interval $I$, then given any $c > 0$ and any $p > 0$, there exists $D > 0$ and a measurable subset $E_I \subset I$ disjoint from $F$ for which we have \begin{equation}
    \label{exactLogPluswInt}
    \int_{E_I} \min \big(p^{-1}\log^+(1/w),D \big)  \, dx = c.
\end{equation}
\end{lem}

\begin{proof} On the set $F$, $\log(w)$ is bounded from below by $\log \delta$. Hence $\int_{I\setminus F} \log w \, dx = -\infty$. Consequently, \[\lim_{D \to +\infty} \int_{I \setminus F} \min\big( p^{-1} \log^+(1/w), D\big) \, dx = \int_{I \setminus F} p^{-1}\log^+(1/w) \, dx = +\infty. \] So for $D$ sufficiently large we will have $\int_{I \setminus F} \min\big( p^{-1} \log^+(1/w), D\big) \, dx > c$, and then by the absolute continuity of the finite positive measure $\min\big( p^{-1} \log^+(1/w), D\big) \, dx$ we may choose a set $E_I \subset I \setminus F$ for which \eqref{exactLogPluswInt} holds.
\end{proof}

We will now construct the sequence in \thref{hnSplittingSequence}. Let $p > 2$, $\{c_n\}_{n \geq 1}$ be some sequence of positive numbers which tends to $0$ slowly enough so that $c_n2^n$ tends to $+\infty$, and let $F$ be as above. Fix some integer $n > 0$, cover $\R$ by a sequence of (say, half-open) disjoint intervals of length $2^{-n}$ and let $\{\ell_k\}_k$ be those intervals for which $|\ell_k \cap F| > 0$. Apply \thref{pieceLemma} with $c = c_n$ to each of the intervals $\ell_k$ to obtain a corresponding constant $D_k > 0$ and a set $E_k := E_{\ell_k} \subset \ell_k$ for which \eqref{exactLogPluswInt} holds. We set $d\mu(t) = \log W(t) \, dt$, where \[ \log W(t) = \sum_k \min\big(p^{-1} \log^+(1/w(t)), D_k \big) \mathbbm{1}_{E_k}(t) - \frac{c_n}{|\ell_k \cap F|}\mathbbm{1}_{\ell_k \cap F}(t).\] Then $\mu$ is an absolutely continuous real-valued measure with bounded density $\log W$, $\mu(\ell_k) = 0$ and $|\mu|(\ell_k) = 2c_n$. We construct $h_n \in \hil^\infty(\R)$ by letting the logarithm $\log h_n(z)$ of its extension to $\H$ be given by the right-hand side of the formula \eqref{hOuterDef}. Then, by \thref{PsnIntegralEstimate}, we have the inequalities  
\[ |\Po_\mu(x+iy)| \leq \frac{2c_n}{\pi y}, \quad z = x+iy \in \H\] and consequently, since $|h(z)| = e^{\Po_\mu(z)}$, we have the bounds
\begin{equation}
    \label{hEq1} e^{-2c_n/\pi y} \leq |h_n(z)| \leq e^{2c_n/\pi y}, \quad z = x+iy \in \H.
\end{equation} 
Since $c_n \to 0$, by multiplying $h_n$ by an appropriate unimodular constant, we may assume by \eqref{hEq1} that \begin{equation}
\label{hEq4}\lim_{n \to \infty} h_n(z) = 1
\end{equation} for every $z \in \H$ (possibly after passing to a subsequence). Also, for almost every $x \in F$, we have 
\begin{equation}
    \label{hEq2} |h_n(x)| = e^{-c_n/|\ell_k \cap F|} \leq e^{-c_n 2^n}.
\end{equation} Our assumption on $c_n$ then implies that $\lim_{n \to \infty} h_n(x) = 0$ for almost every $x \in F$. For almost every $x \in \R \setminus F$, we have instead \begin{equation}
\label{hEq3} |h_n(x)|^p w(x) = e^{p \log W(x)}w(x) \leq e^{\log^+(1/w(x))}w(x) \leq 1.\end{equation} 

The equations \eqref{hEq1}, \eqref{hEq4}, \eqref{hEq2} and \eqref{hEq3} show that the sequence $\{h_n\}_n$ satisfies the conditions stated in \thref{hnSplittingSequence}. Thus \thref{claim2} holds, and consequently the proof of \thref{CondensationTheorem} is complete.

\section{Proof of \thref{SparsenessTheorem}}

\subsection{A bit of concave analysis} Let $M: \R_+ \to \R_+$ be an increasing and concave function which is differentiable for $x > 0$. We assume that $M(0) = 0$, and that \begin{equation}
    \label{MsqrtXest} M(x) \leq \sqrt{x}, \quad x > 0.
\end{equation} For a function $M$ with the above properties, the integrals \begin{equation}
    \label{IMintegrals}
    I_M(y) := \int_0^\infty e^{M(x) - 2yx} \, dx
\end{equation} converge for every $y > 0$, and estimation of the growth of $I_M(y)$ as $y \to 0^+$ will be of importance in the proof of \thref{SparsenessTheorem}. To estimate $I_M$, we define a function $M_*$ in the following way. Since $M$ is increasing, concave and differentiable, the derivative $M'(x)$ is defined for $x > 0$, and it is a positive and decreasing function. The concavity of $M$ implies that \begin{equation} \label{MprimeEstMx} M'(x) \leq  M(x)/x
\end{equation} from which it follows by \eqref{MsqrtXest} that \[ \lim_{x \to \infty} M'(x) = 0.\] It is only the asymptotic behaviour of $M$ as $x \to \infty$ that concerns us, so we will also assume for convenience that $\lim_{x \to 0^+} M'(x) = +\infty$. In this case, the inverse function \begin{equation}
     \label{MprimeInverseDef}
     K(y) := (M')^{-1}(y), \quad y \in (0,\infty)
 \end{equation} is well-defined and positive. It is decreasing, and satisfies $\lim_{y \to 0^+} K(y) = +\infty$. We set 
\begin{equation}
    M_*(y) := M(K(y)), \quad y \in (0,\infty)
\end{equation} The function $M_*$ is decreasing, and satisfies $\lim_{y \to 0^+} M_*(y) = +\infty$. The integrals $I_M(y)$ can be estimated in terms of $M_*$.

\begin{prop} \thlabel{IMestimate}
    For $M$ as above, we have \[ I_M(y) \leq \frac{2 e^{M_*(y)}}{y^2}\] for all sufficiently small $y > 0$.
\end{prop}

\begin{proof}
    We will need the following observation. The inequality \eqref{MprimeEstMx} together with \eqref{MsqrtXest} implies that $ M'(x) \leq \frac{1}{\sqrt{x}}$ if $x > 0$ is sufficiently large. We want to show the inequality \[ K(y) \leq \frac{1}{y^2}, \] for sufficiently small $y > 0$. Set $K(y) = x$ and $\frac{1}{y^2} = x_*$. Then \[ M'(x) = y = \frac{1}{\sqrt{x_*}} \geq M'(x_*)\] if $y > 0$ is sufficiently small (and consequently $x_*$ is sufficiently large). Since $M'$ is a decreasing function, the above inequality shows that $x_* \geq x$, which is the same as the desired inequality.
    
    We split the integral \eqref{IMintegrals} at $x = K(y)$. Both pieces can be estimated very crudely. For the first piece, we have \begin{align*}
        \int_0^{K(y)} e^{M(x) - 2yx} \, dx &\leq \int_0^{K(y)} e^{M(K(y))} \, dx \\ 
        &= K(y) e^{M_*(y)} \\
        &\leq \frac{e^{M_*(y)}}{y^2}.
    \end{align*} We used our initial observation in the last step. For the second piece, we note that \[ \sup_{x > 0}\,  M(x) - yx = M(K(y)) - yK(y) \leq M_*(y).\] Indeed, the supremum is attained at the point $x$ where $M'(x) = y$, which by definition is $x = K(y)$. Thus \begin{align*}
        \int_{K(y)}^\infty e^{M(x) - 2yx} \, dx \leq&  \int_{K(y)}^\infty e^{M_*(y) - xy} \, dx \\
        \leq& e^{M_*(y)} \int_0^\infty e^{-yx} \, dx \\
        = & \frac{e^{M_*(y)}}{y} \\
         \leq & \frac{e^{M_*(y)}}{y^2}.
    \end{align*} In the last inequality we require that $y \in (0,1)$. We obtain the desired estimate by combining the estimates for the two pieces of the integral.
\end{proof}

In the proof of \thref{SparsenessTheorem}, a point will come up where we will need integrability of $M_*$ near the origin. The next proposition describes the functions $M$ corresponding to $M_*$ which are integrable in this way.

\begin{prop} \thlabel{MstarIntegrabilityProp} For $M$ as above, the following two statements are equivalent.
\begin{enumerate}[(i)]
    \item $\int_0^{\delta} M_*(y) \, dy < \infty$ for some $\delta > 0$.
    \item $\int_1^{\infty} (M'(x))^2 \, dx < \infty$.   
\end{enumerate}    
\end{prop}

\begin{proof}
    We start with the integral in $(i)$ above and implement the change of variable $y = M'(x)$. Then $dy = M''(x) \, dx$ and $M_*(y) = M(x)$, and by next using integration by parts, we obtain \begin{align*}
        \int_0^{\delta} M_*(y) \, dy &= -\int_{K(\delta)}^{\infty} M(x) M''(x) \, dx  \\ &= \lim_{R \to +\infty} -M(R)M'(R) + M(K(\delta))M'(K(\delta)) \\ &+  \lim_{R \to +\infty} \int_{K(\delta)}^{R} (M'(x))^2 \, dx.
    \end{align*} The assumptions \eqref{MsqrtXest} and \eqref{MprimeEstMx} together imply that the quantity $M'(R)M(R)$ stays bounded, as $R \to +\infty$. Thus the desired equivalence follows from the above computation.
\end{proof}

\begin{remark} \thlabel{ConcFuncRemark} We should point out that a concave function indeed exists which satisfies all of our conditions. For instance, note that any concave function $M$ which for large positive $x$ coincides with \[ M(x) := \frac{\sqrt{x}}{\log x}\] satisfies the equivalent conditions of the above proposition. Indeed, one verifies by differentiation that the right-hand side is concave on some interval $[A, \infty)$ and that the integral in $(ii)$ of \thref{MstarIntegrabilityProp} converges. Such a function $M$ can be easily chosen to satisfy all assumptions made in this section. Moreover, clearly it satisfies \eqref{MfastIncrease}.
\end{remark}

\subsection{A growth estimate for Hardy class functions}

\begin{prop} \thlabel{growthestimateUpperHalfplane}
    Let $f \in \hil^1(\R) \cap \hil^2(\R)$ have a Fourier transform $\widehat{f}$ satisfying \[ \int_{\R_+} |\widehat{f}|^2 \rho \, d\zeta \leq C\] for some constant $C > 0$ and $\rho(\zeta) = e^{-M(\zeta)}$. Then there exist a positive constant $\delta$ such that the analytic extension of $f$ to $\H$ satisfies the estimate \[ |f(x+iy)| \leq \sqrt{2 C}  \frac{e^{M_*(y)}}{y}, \quad y \in (0, \delta).\]
\end{prop}

\begin{proof}
    By developments of Section \ref{HardyClassSubsection}, we have \[ f(z) = \int_\R \frac{f(t)}{t-z} \frac{dt}{2 \pi i } = \int_\R f(t) \overline{\psi_z(t)} d\lambda(t), \quad z = x+iy \in \H,\] and where $\psi_z \in \hil^2(\R)$ is as in \thref{FourierTransformCauchyKernel}. By Plancherel's theorem and the lemma, we obtain
    \begin{align*}
       f(z) & = \int_{\R_+} \widehat{f}(\zeta)e^{i z \zeta} \, d\lambda(\zeta) \\
       &= \int_{\R_+} \widehat{f}(\zeta) \sqrt{\rho(\zeta)} \frac{e^{-y\zeta + ix\zeta}}{\sqrt{\rho(\zeta)}} \, d\lambda(\zeta)
    \end{align*} An application of Cauchy-Schwarz inequality leads to 
    \begin{equation} \label{sect62equation1}
        |f(z)| \leq \sqrt{C} \sqrt{ \int_0^\infty e^{M(\zeta) -2y\zeta}\, d\zeta} .
    \end{equation} 
    Now \thref{IMestimate} applies to obtain the desired estimate.
\end{proof}

\subsection{Construction of the compact set} If $M$ satisfies the equivalent conditions of \thref{MstarIntegrabilityProp}, then the logarithm of the right-hand side in the inequality of \thref{growthestimateUpperHalfplane}, namely \begin{equation}
    \label{Hdef} H(y) := \frac{\log (2C)}{2} + M_*(y) - \log y, \quad y \in (0, \delta),
\end{equation} is, for small enough $\delta > 0$, positive and integrable over the interval $y \in (0, \delta)$. To $H$ and any $A > 0$ we will associate a Cantor-type compact set $E$ contained in $[0,A]$ which contains no intervals and for which the integral 
\begin{equation}
    \label{HdistConvInt}
    \int_{[0,A] \setminus E} H(\dist{x}{E}) \, dx
\end{equation} converges. Here $\dist{x}{E}$ denotes the distance from the point $x \in [0,A]$ to the closed set $E$. Let $\mathcal{U} = \{ \ell \}$ be the system of maximal disjoint open intervals, union of which constitutes the complement of $E$ within $(0,A)$. The convergence of the integral above is easily seen to be equivalent to the convergence of the sum 
\begin{equation}
\label{HcompInvConvInt} \sum_{\ell \in \mathcal{U}} \int_0^{|\ell|} H(x) \, dx \end{equation} 
where $|\ell|$ denotes the length of the interval $\ell$. 
\label{SetESubsection} To construct $E$, we choose a sequence of numbers $\{L_n\}_{n \geq 1}$ which is so quickly decreasing that \begin{equation}
    \label{LnsummationCond0}
    \sum_{n = 1}^\infty 2^n \int_0^{L_n} H(x) \, dx < \infty 
\end{equation} and 
\begin{equation}
    \label{LnsummationCond}\sum_{n = 1}^\infty 2^nL_n \leq A/2.
\end{equation} From such a sequence, we construct $E$ as in the classical Cantor set construction. We set $E_0 = [0, A]$, and recursively define a compact set $E_{n+1}$ contained in $E_n$. The set $E_{n+1}$ consists of $2^{n+1}$ closed intervals in $[0,A]$ which we obtain by removing from the $2^n$ closed intervals $\{E_{n,i}\}_{i=1}^{2^n}$ constituting $E_n$ an open interval of length $L_n$ lying in the middle of $E_{n,i}$. Thus splitting each $E_{n,i}$ into two new closed intervals. The above summation condition \eqref{LnsummationCond} ensures that $|E_n| > A/2$, and so $E := \cap_{n=0}^\infty E_n$ has positive Lebesgue measure which is not less than $A/2$. The integral condition \eqref{HdistConvInt} holds by its equivalence to \eqref{HcompInvConvInt} and by \eqref{LnsummationCond0}. Clearly $E$ contains no intervals. 

\subsection{Collapse of the Fourier transforms}

We are now ready to prove \thref{claim3}. We will do so by showing that $\hil(\mathbbm{1}_E, \rho)$ does not contain any non-zero tuple of the form $(0, k)$, $k \in \L^2(\R_+, \rho \, dx)$, where $M$ is as in \thref{ConcFuncRemark}, for instance, and where $E$ is as in Section \ref{SetESubsection}. We set $\rho = e^{-M}$.

\begin{lem} \thlabel{FourierCollapseLemma}
Let $E$, $M$and $\rho$ be chosen as above. Assume that $\{f_n\}_n$ is a sequence of functions in $\hil^1(\R) \cap \hil^2(\R)$, each of which has an analytic extension to a half-space larger than $\H$. If $\lim_{n \to \infty} f_n = 0$ in the norm of $\L^2(E, \mathbbm{1}_E \, dx)$ and the sequence of Fourier transforms $\{\widehat{f_n}\}_n$ satisfies \[ \sup_{n} \int_{\R_+} |\widehat{f_n}|^2 \rho \, d\zeta < \infty, \] then we have \[ \lim_{n \to \infty} f_n(z) = 0, \quad z \in \H.\] The convergence is uniform on compact subsets of $\H$. 
\end{lem}

In the proof of \thref{FourierCollapseLemma} given below we will use a technique of Khrushchev from \cite{khrushchev1978problem} for estimating harmonic measures on certain domains. For general background on the theory of harmonic measures, see \cite{garnett2005harmonic} or \cite{ransford1995potential}. 

Let $\mathcal{U} = \{ \ell \}$ be the collection of finite open intervals complementary to $E$, and let \[T_\ell := \{x+iy \in \H : x \in \ell, y \leq \dist{x}{E} \}\] be a triangle with base at $\ell$. We define $\Omega = R \setminus \big( \cup_{\ell \in \mathcal{U}} T_{\ell} \big)$ to be the bounded domain in $\H$ which consists of a rectangle $R$, with a base being the shortest closed interval containing the set $E$, with the triangles $T_\ell$ removed from $R$. See Figure \ref{fig:OmegaDomain}. An observation that Khrushchev made regarding this type of domains is the following property of their harmonic measure.

\begin{lem} \thlabel{KhruschevsEstimateLemma}
    Let $E$, $M$ and $\rho$ be chosen as above, and let $H$ be given by \eqref{Hdef}. Let $\Omega$ be the domain described above. If $\omega_z$ is the harmonic measure of the domain $\Omega$ at any point $z \in \Omega$, then \[ \int_{\partial \Omega \cap \H} H(\Im t) \, d\omega_z(t) < \infty.\]
\end{lem}

We emphasize that $\partial \Omega \cap \H$ equals $\partial \Omega \setminus \R$.

\begin{proof}
    The proof is very similar to the one given by Khrushchev in \cite{khrushchev1978problem}, only minor details differ. 
    If $\ell = (a,b)$ is one of the finite intervals complementary to $E$, and $T_\ell$ is the triangle standing on top of it, then we denote by $A(s)$ the part of the boundary of $T_\ell$ which lies above the interval $(a, a+s) \subset \R$, $0 < s < |\ell|/2$. If $u$ is the harmonic measure in $\H$ of the interval $(a, a+s)$, then it is easy to see from the explicit formula \[ u(x+iy) = \frac{1}{\pi}\int_a^{a+s} \frac{y}{(x-t)^2 + y^2}\, dt, \quad x+iy \in \H\] that $u(x+iy) \geq \frac{1}{2 \pi}$ for $x+iy \in A(s)$. Since $u$ is harmonic and continuous in the closure of $\Omega$ except possibly at the two points $a$ and $a+s$, the reproducing formula $\int_{\partial \Omega} u \, d\omega_z = u(z)$ holds, and so \[ \omega_z(A(s)) = \int_{A(s)} d\omega_z \leq \int_{\partial \Omega} 2\pi u \, d\omega_z = 2\pi u(z) \leq \frac{2 s}{ y}, \quad z = x+iy \in \H.\] We have used the positivity of $u$ and $\omega_z$ in the first inequality, and the second one is an easy consequence of the explicit formula for $u$ above. Set $A_1 := A(|\ell|/2)$, which is the left side of the boundary of $T_\ell$, and further set $A_n := A(|\ell|/2^n)$, $n \geq 1$. Then, since $H$ is decreasing, \begin{align*}
        \int_{\partial T_\ell \cap \H} H(\Im t) d\omega_z(t) &= 2 \int_{A_1} H(\Im t) \, d\omega_z(t) \\
        & \leq 2 \sum_{n = 1}^\infty \int_{A_n - A_{n+1}} H(|\ell|/2^{n+1}) \, d\omega_z(t) \\
        &\leq 2 \sum_{n = 1}^\infty H(|\ell|/2^{n+1})\frac{2|\ell|}{2^n y} \\
        &\leq \frac{16}{y} \int_0^{|\ell|} H(t)\, dt.
    \end{align*} Now the desired claim follows from \eqref{HcompInvConvInt}.
\end{proof}

\begin{figure}
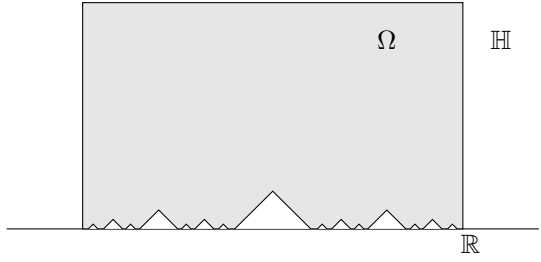

    \centering
    \includestandalone[scale=1]{OmegaDomain}
     \caption{The domain $\Omega$ in the proof of \thref{FourierCollapseLemma}. There is a triangular tent between $\Omega$ and each complementary interval of $E$, and $E$ lives on $\R$ inbetween the tents.}
    \label{fig:OmegaDomain}
\end{figure}

\begin{proof}[Proof of \thref{FourierCollapseLemma} ] 
Note that it is sufficient to establish the claim that the sequence $\{f_n\}_n$ contains a subsequence which converges pointwise in $\Omega$ to $0$. Indeed, the proof of \thref{growthestimateUpperHalfplane} shows that our assumption on the Fourier transforms $\widehat{f}$ implies pointwise boundedness of the sequence $\{f_n\}_n$ on each half-plane $\{ x+iy \in \H: y > \delta\}$, $\delta > 0$. Hence the sequence $\{f_n\}_n$ forms a normal family on $\H$. If we establish the above claim, then every subsequences of $\{f_n\}_n$ contains a further subsequence convergent to $0$ in $\H$. This is equivalent to convergence of the entire initial sequence $\{f_n\}_n$ to $0$.

Fix $z \in \Omega$. Since $\log|f_n|$ is a subharmonic function and $\max(-N, \log|f_n|)$ is a bounded continuous function on $\partial \Omega$, we obtain by the maximum principle for subharmonic functions that \[ \log |f_n(z)| \leq \int_ {\partial \Omega} \max(-N, \log |f_n(t)|) \, d\omega_z(t).\]  We let $N \to +\infty$ and, by the monotone convergence theorem, obtain \begin{align*}
    \log|f_n(z)| &\leq \int_{\partial \Omega} \log|f_n(t)| \, d\omega_z(t) \\  &= \int_{\partial \Omega \cap \H} \log|f_n(t)| \, d\omega_z(t) + \int_{E} \log|f_n(t)| \, d\omega_z(t)
\end{align*} The assumption in the lemma, \thref{growthestimateUpperHalfplane}, the definition of $H$ in \eqref{Hdef} and \thref{KhruschevsEstimateLemma} show that \[\int_{\partial \Omega \cap \H} \log|f_n(t)| \, d\omega_z(t) \leq \int_{\partial \Omega \cap \H} H(\Im t) d\omega_z(t) < A \] where $A$ is some positive constant which is independent of $n$. By Egorov's theorem, we may pass to a subsequence (the same subsequence for each $z \in \Omega)$ and assume that $f_n$ converge uniformly to $0$ on some subset $E'$ of $E$ which is of positive Lebesgue measure. On $E \setminus E'$ we have the estimate \[ \int_{E \setminus E'} \log|f_n(t)| \, d\omega_z(t) \leq \int_{E \setminus E'} |f_n(t)|^2 \, d\omega_z(t) \leq \frac{1}{\pi \Im z}\int_{E \setminus E'} |f_n(t)|^2 \, dx. \] The last inequality follows from monotonicity of the harmonic measure with respect to domains (see \cite[Corollary 4.3.9]{ransford1995potential}), which applied to $\Omega \subset \H$ leads to the inequality \[ \omega_z(B) \leq \int_B \Po(t,z) \, dt \leq \frac{|B|}{\pi \Im z}\]  for any Borel subset $B$ of $E$. Thus $d \omega_z \leq \frac{dx}{\pi \Im z}$, attesting the integral inequality above. By the convergence of $f_n$ to $0$ in the norm of $\L^2(E, \mathbbm{1}_E\, dx)$, the integrals $\int_{E \setminus E'} |f_n|^2 dx$ are uniformly bounded by some constant $C > 0$, and so the above inequalities give \[ \log |f_n(z)| \leq A + C + \int_{E'} \log |f_n(t)| d\omega_z(t).\] But $\omega_z(E') > 0$, since the harmonic measure and the arc-length measure on the rectifiable curve $\partial \Omega$ are mutualy absolutely continuous (see \cite[Theorem 1.2 of Chapter VI]{garnett2005harmonic}), so since $|f_n|$ converge uniformly to $0$ on $E'$, the integral on the right-hand side above converges to $-\infty$ as $n \to \infty$. Thus $|f_n(z)| \to 0$, and since $z \in \Omega$ was arbitrary, the desired claim follows.    
\end{proof}

The following proposition implies \thref{claim3} of Section \ref{StrategySection}, and so it also implies \thref{SparsenessTheorem}.

\begin{prop} \thlabel{FourierCollapseProposition}
Let $E$, $M$and $\rho$ be chosen as above. If a tuple of the form $(0,k) \in \L^2(E, \mathbbm{1}_E\, dx) \oplus \L^2(\R_+, \rho)$ is contained in the Hardy subspace $\hil(\mathbbm{1}_E, \rho)$, then $k \equiv 0$.
\end{prop}

\begin{proof}
    By the containment $(0,k) \in \hil(\mathbbm{1}_E, \rho)$ and \thref{KernelContainmentHardySubspace}, there exists a sequence $\{f_n\}_n$ of functions in $\hil^1(\R) \cap \hil^2(\R)$ extending analytically across $\R$ and for which the tuples $Jf_n = (f_n, \widehat{f}_n)$ converge in the norm of the product space $\L^2(\R, \mathbbm{1}_E\, dx) \oplus \L^2(\R_+, \rho \, dx)$ to $(0,k)$. By passing to a subsequence, we may assume that the Fourier transforms $\widehat{f_n}$ converge pointwise almost everywhere on $\R_+$ to $k$. One might attempt to prove the proposition by using the formula in \thref{H1FourierTransformFormula}, and observing that \[ k(\zeta) = \lim_{n \to \infty} \widehat{f_n}(\zeta) = \lim_{n \to \infty} e^{y\zeta} \int_{\R} f_n(x+iy) e^{-ix\zeta} \, d\lambda(x)\] holds for almost every $\zeta \in \R_+$ and for any $y > 0$. By \thref{FourierCollapseLemma} the integrand converges pointwise to $0$. However, an appeal to the usual convergence theorems for integrals is not justified, and we have to proceed more carefully. Note that since $k \in \L^2(\R_+, \rho \, dx)$ and $\rho$ is bounded from below on compact subsets of $\R_+$, in fact $k$ is locally integrable on $\R$. It follows that we can interpret $k$ as a distribution on $\R_+$. Thus to show that $k \equiv 0$, it suffices to establish that $\int_{\R_+} k \phi \,d\lambda = 0$ for every smooth function $\phi$ which is compactly supported in $\R_+$. 
    
    Let $\phi$ be as above. Since we have that $\widehat{f_n} \to k$ in $\L^2(\R_+, \rho \, dx)$ and $\rho$ is bounded from below on compact subsets of $\R_+$, we obtain \begin{equation}
        \label{kDistConv} \int_{\R_+} k\phi \,d\lambda = \lim_{n \to \infty} \int_{\R_+} \widehat{f_n} \phi \, d\lambda.
    \end{equation} Fix some small $y > 0$. By \thref{H1FourierTransformFormula}, we get that \[  \widehat{f_n}(\zeta) = e^{y\zeta} \int_{\R} f_n(x+iy) e^{-ix\zeta} \, d\lambda(x).\] Plugging this formula into \eqref{kDistConv} and noting that the use of Fubini's theorem is permitted, we obtain 
    \begin{align}
        \int_{\R_+} k\phi \,d\lambda &= \lim_{n \to \infty} \int_\R \Big(\int_{\R_+} \phi(\zeta)e^{y \zeta} e^{-ix\zeta} d\lambda(\zeta) \Big) f_n(x+iy) \, d\lambda(x) \nonumber \\
        & = \lim_{n \to \infty} \int_\R D(x) f_n(x+iy) \, d\lambda(x) \label{kphiIntegral}
    \end{align} where \[ D(x) := \int_{\R_+} \phi(\zeta)e^{y \zeta} e^{-ix\zeta} d\lambda(\zeta)\]  is the Fourier transform of the compactly supported smooth function $\zeta \mapsto \phi(\zeta)e^{y\zeta}$. As such, $D$ is certainly integrable on $\R$. By \thref{FourierCollapseLemma}, we have $\lim_{n \to \infty} f_n(x+iy) = 0$, and \[ \sup_{n} \, \sup_{x \in \R} |f_n(x+iy)| < \infty\] holds by \thref{growthestimateUpperHalfplane}. Therefore, this time, the dominated convergence theorem applies to \eqref{kphiIntegral}, and we conclude that \[ \int_{\R_+} k\phi \, d\lambda = 0.\] Thus $k$ is the zero distribution on $\R_+$, and therefore $k \equiv 0$.
\end{proof}

\section{Clumping for tempered distributions}

In this last section, we indicate how one can derive \thref{DistributionalClumpingTheorem} from \thref{CondensationTheorem}. We will skip most of the details of the necessary computations, which are in any case standard.

Let $f$ be a function which satisfies \begin{equation} \label{fTempGrowthEq} \int_\R \frac{|f(x)|}{(1+|x|)^n} \, dx < \infty \end{equation} for some positive integer $n$. Then $f$ can be interpreted as a tempered distribution on $\R$ in the usual way, and so $f$ has a distributional Fourier transform $\widehat{f}$. Our hypothesis is that $\widehat{f}$ is an integrable function on some half-axis $[\zeta_0, \infty)$ and that \begin{equation}
    \label{hatfDecayS7}\rho_{\widehat{f}}(\zeta) = \mathcal{O}\big( e^{-c\sqrt{\zeta}}\big), \quad \zeta > \zeta_0.
\end{equation} We may assume that $\zeta_0 = 0$. In order to prove \thref{DistributionalClumpingTheorem}, we will construct an appropriate multiplier $m: \R \to \mathbb{C}$ with the property that $mf$ is a function to which \thref{CondensationTheorem} applies. In particular, the following properties will be satisfied by $m$:

\begin{enumerate}[\itshape (i)]
    \item $m(x)$ is a bounded function of $x \in \R$ which is non-zero for almost every $x \in \R$.
    \item $mf \in \L^2(\R, dx)$,
    \item $\rho_{\widehat{mf}}(\zeta) = \mathcal{O}\big(e^{-c\sqrt{\zeta}}\big)$ for some $c > 0$ and $\zeta > 0$,
    \item $\int_I \log |m| \, dx > - \infty$ for every interval $I \subset \R$.
\end{enumerate}

If we construct such a multiplier $m$, then $(ii)$, $(iii)$ and \thref{CondensationTheorem} imply that $\log |m f|$ is locally integrable on an open set $U$ which coincides, up to a set of measure zero, with $\{ x \in \R : |f(x)m(x)| > 0 \}$. By $(i)$, $U$ differs from $\{ x \in \R : |f(x)| > 0 \}$ at most by a set of measure zero. Moreover, the formula $\log |f| = \log |f m| - \log |m|$ and $(iv)$ show that $\log |f|$ is locally integrable on $U$. This proves \thref{DistributionalClumpingTheorem}, as a consequence of existence of a multiplier satisfying the above conditions. We now show how to construct such a multiplier.

We set \[\Phi(x) := \frac{(-i)^n n!}{\sqrt{2\pi} (x-i)^n}\] and let $h$ be defined by the equation \eqref{hOuterDef}, with \[\log |h(x)| = \log \min (1, |f(x)|^{-1}), \quad x \in \R. \] The condition \eqref{fTempGrowthEq} ensures that $h$ is well-defined, and it is a member of $\hil^\infty(\R)$. We put \[ h_*(x) := \frac{h(x)}{(x+i)^2}, \quad x \in \R. \] and finally \[ m(x) := \Phi(x) \overline{h_*(x)}, \quad x \in \R. \] Clearly, $m$ is bounded. Since $|f|$ is locally integrable on $\R$, the set $\{ x \in \R : |f| = \infty \}$ has measure zero. Consequently, $|h_*| > 0$ almost everywhere in $\R$, and so the desired property $(i)$ of $m$ holds. The choice of $h_*$ and $\Phi$ ensures that $mf$ is both bounded and integrable on $\R$, implying $mf \in \L^2(\R, dx)$, so that $(ii)$ above holds. Property $(iv)$ holds by \thref{HadyClassLogIntProp}, since $h_* \in \H^1(\R)$. So the critical property left to be verified is the spectral estimate of $mf$ in $(iii)$ above.

\begin{lem} 
    With notation and definitions as above, the Fourier transform $\widehat{mf}$ satisfies \[ \rho_{\widehat{mf}}(\zeta) = \mathcal{O}\big( e^{-c\sqrt{\zeta}} \big), \quad \zeta > 0\] for some $c > 0$.
\end{lem}

\begin{proof} A standard argument shows that $\widehat{f \Phi}$ must coincide on $\R_+$ with the convolution $\widehat{f} \ast \widehat{\Phi}$ (which, note, is a function on $\R_+$). Indeed, let $s$ be a Schwartz function which has a Fourier transform $\widehat{s}$ supported on some compact interval $[a,b]$, $0 < a < b$. Note that the function $\Phi s$ is also of Schwartz class. It follows immediately from the integral definition of the Fourier transform \eqref{FourierTransformDef} that $\widehat{\overline{\Phi}s} = \widehat{\conj{\Phi}} \ast \widehat{s}$, and that $\widehat{\overline{\Phi}s}$ is supported on the interval $[a, \infty)$. Hence, by the definition of the distributional Fourier transform, we obtain \[ \int_{\R_+} \widehat{f \Phi} \, \overline{\widehat{s}}\, d\lambda = \int_\R f\Phi \overline{s} \, d\lambda = \int_\R f \overline{ \overline{\Phi} s} \, d\lambda = \int_{\R_+} \widehat{f} \overline{(\widehat{\overline{\Phi}} \ast \widehat{s})} \, d\lambda.\] Fubini's theorem and the computational rule $\widehat{\conj{\Phi}}(x) = \conj{\widehat{\Phi}}(-x)$ shows that the last integral above equals \[\int_{\R_+} (\widehat{f} \ast \widehat{\Phi})  \, \overline{\widehat{s}} \, d\lambda,\] proving our claim about the structure of $\widehat{f \Phi}$ on $\R_+$. 

Hence $\widehat{f\Phi}$ is a bounded continuous function which coincides with \[ \widehat{f \Phi}(\zeta) = \int_{\R} \widehat{f}(x) \widehat{\Phi}(\zeta - x) \, d\lambda(x) \] for $\zeta > 0$. By a computation similar to the one in the proof of \thref{FourierTransformCauchyKernel} one sees that $\Phi$ has the Fourier transform \[ \widehat{\Phi}(\zeta) = |\zeta|^n e^{\zeta} \mathbbm{1}_{\R_-}(\zeta).\] For such $\zeta$, we estimate \begin{align*}
    \big\vert \widehat{f \Phi}(\zeta) \big\vert & \leq \int_{\R} |\widehat{f}(x)| |\zeta - x|^n e^{\zeta - x} \mathbbm{1}_{\R_-}(\zeta - x) \, d\lambda(x) \\
    &= \int_\zeta^\infty |\widehat{f}(x)| |\zeta - x|^n e^{\zeta - x} \, d\lambda(x) \\ & = \sum_{k=0}^\infty \int_{\zeta 2^k}^{\zeta 2^{k+1}} |\widehat{f}(x)| |\zeta - x|^n e^{\zeta - x} \, d\lambda(x).
\end{align*} We now make the rather rough estimate \[ |\zeta - x|^n e^{\zeta - x} \leq \zeta^{2(k+1)n}, \quad x \in [\zeta 2^k, \zeta 2^{k+1}],\] which gives \[\big\vert \widehat{f \Phi}(\zeta) \big\vert \leq \sum_{k=0}^\infty \rho_{\widehat{f}}(\zeta 2^k) \zeta^{2(k+1)n} \leq A \sum_{k=0}^\infty e^{-c \sqrt{\zeta} \sqrt{2}^k} \zeta^{2(k+1)n}\] for some $A > 0$. The above sum can be readily estimated to be of order $\mathcal{O}\big( e^{-d \sqrt{\zeta}})$ for some $d > 0$ slightly smaller than $c$. Since $fm$ is the product of two integrable functions $f\Phi$ and $\overline{h_*}$, we have \[\widehat{fm} = \widehat{f\Phi} \ast \widehat{\overline{h_*}}, \] where $\widehat{\overline{h_*}}(\zeta) = \overline{\widehat{h_*}(-\zeta)}$ is non-zero only for $\zeta < 0$. Note that $h_*$ is integrable on $\R$, and so $\widehat{h_*}$ is bounded. We obtain 
    \begin{align*}
        |\widehat{fm}(\zeta)| &\leq \int_\R \Big\vert \widehat{f\Phi}(x) \overline{\widehat{h_*}(x-\zeta)}\Big\vert \, d\lambda(x) \\ &\leq B  \int_\zeta^\infty e^{-d\sqrt{x}} \, d\lambda(x) = \mathcal{O}\big( e^{-d  \sqrt{\zeta}}\big),    
    \end{align*} where $B$ is some positive constant. The desired estimate on $\rho_{\widehat{fm}}$ follows readily from this estimate.
\end{proof}
By the above discussion, we have proved \thref{DistributionalClumpingTheorem}.

\bibliographystyle{plain}
\bibliography{mybib}

\end{document}

%% file: OmegaDomain.tex
\begin{tikzpicture}
    \draw[] (-1,0.0) -- (6,0.0);
    \filldraw[fill=gray!20] (0,0) -- (5,0) -- (5,3) -- (0,3) -- cycle;
    \node [black] at (5.1, -0.2) {$\mathbb{R}$};
    \node [black] at (5.5, 2.5) {$\mathbb{H}$};
    \node [black] at (4, 2.5) {$\Omega$};
    \filldraw[fill=white] (2, 0) -- (2.5, 0.5) -- (3, 0);

    \foreach \x in {1, 4}
        \filldraw[fill=white] (\x-0.25, 0) -- (\x , 0.25) -- (\x+0.25, 0);

    \foreach \x in {0.4, 1.6, 3.4, 4.6}
        \filldraw[fill=white] (\x-0.125, 0) -- (\x , 0.125) -- (\x+0.125, 0);

     \foreach \x in {0.1375, 0.63, 1.3625, 1.85, 5-1.85, 5-1.365, 5-0.63, 5-0.1375}
        \filldraw[fill=white] (\x-0.0625, 0) -- (\x , 0.0625) -- (\x+0.0625, 0);

\end{tikzpicture}

%% file: ClumpsNStuff.bbl
\begin{thebibliography}{10}

\bibitem{amrein1977support}
W.~O. Amrein and A.~M. Berthier.
\newblock On support properties of $l_p$-functions and their {F}ourier
  transforms.
\newblock {\em Journal of Functional Analysis}, 24(3):258--267, 1977.

\bibitem{benedicks1985fourier}
M.~Benedicks.
\newblock Fourier transforms of functions supported on sets of finite
  {L}ebesgue measure.
\newblock {\em Journal of Mathematical Analysis and Applications},
  106(1):180--183, 1985.

\bibitem{borichev1990uniqueness}
A.~Borichev and A.~Volberg.
\newblock Uniqueness theorems for almost analytic functions.
\newblock {\em Leningrad Mathematical Journal}, 1:157--190, 1990.

\bibitem{garnett}
J.~Garnett.
\newblock {\em Bounded analytic functions}, volume 236.
\newblock Springer Science \& Business Media, 2007.

\bibitem{garnett2005harmonic}
J.~Garnett and D.~Marshall.
\newblock {\em Harmonic measure}, volume~2.
\newblock Cambridge University Press, 2005.

\bibitem{havinbook}
V.~P. Havin and B.~J\"oricke.
\newblock {\em The uncertainty principle in harmonic analysis}, volume~72 of
  {\em Encyclopaedia Math. Sci.}
\newblock Springer, Berlin, 1995.

\bibitem{hedenmalm2011heisenberg}
H.~Hedenmalm and A.~Montes-Rodr{\'\i}guez.
\newblock Heisenberg uniqueness pairs and the {K}lein-{G}ordon equation.
\newblock {\em Annals of Mathematics}, pages 1507--1527, 2011.

\bibitem{khrushchev1978problem}
S.~V. Khrushchev.
\newblock The problem of simultaneous approximation and of removal of the
  singularities of {C}auchy type integrals.
\newblock {\em Trudy Matematicheskogo Instituta imeni VA Steklova},
  130:124--195, 1978.

\bibitem{kriete1990mean}
T.~L. Kriete and B.~D. MacCluer.
\newblock Mean-square approximation by polynomials on the unit disk.
\newblock {\em Transactions of the American Mathematical Society},
  322(1):1--34, 1990.

\bibitem{kulikov2023fourier}
A.~Kulikov, F.~Nazarov, and M.~Sodin.
\newblock Fourier uniqueness and non-uniqueness pairs.
\newblock {\em arXiv preprint arXiv:2306.14013}, 2023.

\bibitem{malman2023revisiting}
B.~Malman.
\newblock Revisiting mean-square approximation by polynomials in the unit disk.
\newblock {\em arXiv preprint arXiv:2304.01400}, 2023.

\bibitem{malman2023shift}
B.~Malman.
\newblock Shift operators, {C}auchy integrals and approximations.
\newblock {\em arXiv preprint arXiv:2308.06495}, 2023.

\bibitem{radchenko2019fourier}
D.~Radchenko and M.~Viazovska.
\newblock Fourier interpolation on the real line.
\newblock {\em Publications math{\'e}matiques de l'IH{\'E}S}, 129:51--81, 2019.

\bibitem{ransford1995potential}
T.~Ransford.
\newblock {\em Potential theory in the complex plane}.
\newblock Number~28 in London Mathematical Society Student Texts. Cambridge
  university press, 1995.

\bibitem{volberg1982logarithm}
A.~Volberg.
\newblock The logarithm of an almost analytic function is summable.
\newblock In {\em Doklady Akademii Nauk}, volume 265, pages 1297--1302. Russian
  Academy of Sciences, 1982.

\bibitem{vol1987summability}
A.~Volberg and B.~J{ö}ricke.
\newblock Summability of the logarithm of an almost analytic function and a
  generalization of the {L}evinson-{C}artwright theorem.
\newblock {\em Mathematics of the USSR-Sbornik}, 58(2):337, 1987.

\end{thebibliography}
